\documentclass[12pt]{article}
\usepackage{graphicx}
\usepackage{graphics}
\usepackage{enumerate}
\usepackage{amssymb,amsmath,latexsym,amsthm,amsfonts}
\usepackage{fancyhdr,fancybox}
 \usepackage{xypic}
 \usepackage[all]{xy}
  \usepackage{color}

\usepackage{tikz} 
\usepackage{skak}

\usepackage{hyperref}
\usetikzlibrary{calc,arrows}
\usetikzlibrary{intersections,decorations} 

\usetikzlibrary{through,decorations.text} 

\newcommand{\ZZ}{\mathbb Z}
\newcommand{\RR}{\mathbb R}

\newcommand{\mP}{\mathcal{P}}

\newcommand{\mQ}{\mathcal Q}
\theoremstyle{definition}

\newtheorem{thm}{Theorem}

\newtheorem{prop}[thm]{Proposition}

\newtheorem{lem}[thm]{Lemma}

\newtheorem{rem}[thm]{Remark}

\newcommand{\mD}{\mathcal{D}}

\newcommand{\vsp}{\vspace{0.1cm}}

\begin{document}

\date{}
\author{Mar\'{\i}a Isabel Cortez \hspace{0.1cm} 
\& \hspace{0.1cm} Andr\'es Navas}

\title{Some examples of repetitive, non-rectifiable Delone sets}
\maketitle
 
A Delone set in $\mathbb{R}^d$ is a subset $\mathcal{D}$ that is 
separated and relatively dense in a uniform way. This means that 
there exist positive real numbers $\rho,\varrho$ such that $d(x,y) \geq \rho$ 
for all $x \neq y$ in $\mathcal{D}$, and for each $z \in \mathbb{R}^d$ 
there is $x \in \mathcal{D}$ satisfying $d(x,z) \leq \varrho$. 
Such a set is said to be {\em repetitive} if there is a function 
$R \!: \mathbb{N} \to \mathbb{N}$ so that for every pair of balls 
$B_r, B_R$ of radius $r$ and $R = R(r)$, respectively, we have that 
$B_R \cap \mathcal{D}$ contains a translated copy of $B_r \cap \mathcal{D}$. 

Besides this pure abstract definition, these sets are revelant in Mathematical Physics as 
models of solid materials, especially after the spectacular discovery of quasicrystals in 
the early eighties by Schechtman and his research team \cite{S}.

A Delone set $\mathcal{D}\! \subset\! \mathbb{R}^d$ is said to be {\em rectifiable} if it is 
bi-Lipschitz equivalent to $\mathbb{Z}^d$. This means that there exists a bijection 
$f \!: \mathcal{D} \to \mathbb{Z}^d$ such that, for some constant $L \geq 1$ and 
all $x, y$ in $\mathcal{D}$,
$$\frac{\| x - y\|}{L} \leq \| f(x) - f(y) \| \leq L \| x-y \|.$$  
The question of the existence of {\em non-rectifiable} Delone sets in $\mathbb{R}^d$, $d \!\geq\! 2$, 
was raised by Gromov (with a geometric group-theoretic motivation \cite{gromov}) 
and Furstenberg (with an ergodic-theoretic motivation \cite{BK2}; see also \cite{H}). This was solved 
in the affirmative by Burago and Kleiner in \cite{BK1} and, independently, by McMullen in \cite{mcmullen}. 
Later, in \cite{BK2}, Burago and Kleiner gave a criterium for a Delone set of the plane to be rectifiable. 
This was extended in \cite{ACG} to larger dimensions by Aliste, Coronel and Gambaudo, who applied 
it to show that Delone sets that are {\em linearly repetitive}, {\em i.e.} those for which the 
{\em repetitivity function} $R$ can be taken linear in $r$, are always rectifiable. This 
includes, for instance, the (set of vertices of the) Penrose tiling; see \cite{solomon}. 
They left open the question 
of the existence of (non-linearly) repetitive Delone sets that are non-rectifiable. The aim of 
this work is to answer this in the affirmative in a very strong way.

Repetitivity has a quite transparent geometric meaning. However, it is also relevant 
from the dynamical viewpoint. Indeed, it is straighforward to verify that this condition is 
equivalent to that the translation action of $\mathbb{R}^d$ on the closure of the orbit of the 
Delone set (endowed with an appropriate Gromov-Hausdorff metric or the 
Chabauty topology) is minimal. In this direction, 
our construction can be further refined to obtain not only minimality but also unique ergodicity, 
which is a much stronger property in this setting. Indeed, a result of Solomyak \cite{solomyak} 
roughly states that, in case of repetitivity, the latter condition is equivalent to that each patch of 
the set not only appears in every big-enough ball, but the number of ocurrences converges 
(as the radius of the ball goes to infinity, independently of the center) to a certain frequency.

\vspace{0.4cm}

\noindent{\bf Main Theorem.} {\em For each $d \geq 2$, there exists 
a subset of $\mathbb{Z}^d$ that is a repetitive, non-rectifiable 
Delone set for which the $\mathbb{R}^d$-action on the closure 
of its orbit is uniquely ergodic.} 

\vspace{0.4cm}

As in \cite{BK1}, in order to avoid technical difficulties mostly concerning notation, 
we will carry out the explicit construction just for the case $d\!=\!2$. (The general 
case proceeds analogously.) We strongly use the main idea of \cite{BK1}, though we 
need to proceed more carefully to get a Delone subset of $\mathbb{Z}^2$ (this is 
the easy part; compare \cite{garber,magazinov}), to guarantee repetitivity (this is 
much more tricky), and finally to ensure unique ergodicity (this is the most technical 
issue). To do this, we develop discrete analogues of the arguments of \cite{BK1} that 
are of independent interest, thus giving a proof of the main result of \cite{BK1} that is 
completely combinatorial ({\em i.e.} without passing to continuous models and/or approximating 
them by discrete ones). In this view, computations involving Jacobians become elementary 
counting arguments, whereas area estimates become density bounds for certain sets.  An 
important advantage of this approach is that it allows giving explicit estimates (and not only 
existencial results) all along the text. In particular, a backtracking  of the estimates of proof 
reveals a quite striking fact: given any unbounded function $R'$, there is a repetitive, 
non-rectifiable Delone set for which the repetitivity function $R$ satisfies $R (r_k) \leq R' (r_k)$ 
along an infinite sequence of radii $r_k \to \infty$. Our method also gives estimates for the speed 
of growing of the sequence $r_k$ provided $R'$ grows faster than linearly. 
This is in contrast to the aforementioned result 
of \cite{ACG}, according to which we cannot have $R(r) \preceq r$ for a non-rectifiable, repetitive 
Delone set. Actually, in our examples, linear repetitivity clearly arises as an obstruction 
for a Delone set to be non-rectifiable. Indeed, along the construction, we need to perform 
modifications that ensure non-rectifiability but that, after rescaling, become 
negligible in density. However, in case of linear repetitivity, the density of 
points where these modifications should be performed persists under scale changes.

The method of construction is still flexible in many ways. In order to illustrate this, recall 
that by a standard application of the ergodic decomposition, the set of invariant probability 
measures of an $\mathbb{R}^d$-action is a {\em Choquet simplex} (that is, a compact, 
convex, metrizable subset of a locally-convex real vector space such that every point 
therein is the mean with respect to a unique probability measure supported on its subset 
of extreme points). In the last paragraph of this paper, we show (the $d\!=\!2$ case of) the 
next extension of our main result (the case of larger dimension $d$ is straightforward 
and left to the reader).

\vspace{0.4cm}

\noindent{\bf Main Theorem (extended).} {\em For each $d \geq 2$ and any 
Choquet simplex $\mathcal{K}$, there exists a subset of $\mathbb{Z}^d$ that is a 
repetitive, non-rectifiable Delone set for which the $\mathbb{R}^d$-action on the closure 
of its orbit has a set of invariant probability measures isomorphic to $\mathcal{K}$.} 

\vspace{0.45cm}

\noindent{\bf I.} {\bf {\em Non-expansiveness implies coarse differentiability.}} As usual, 
for a real number $A$, we denote its integer part by $[A]$. Given two real numbers $A \leq B$, 
we denote $[\![ A,B]\!]$ the set of integers $n$ such that $A \leq n \leq B$. Given 
positive integers $M,N$, we let $R_{M,N} := [\![0,2MN]\!] \times [\![0,M]\!]$. Given 
$k \in [\![ 1,2N ]\!]$ and a positive integer $P$ dividing $M$, let $S_k^P$ be 
the subset of $R_{M,N}$ formed by the points of the form 
\begin{equation}\label{def-points}
x_{i,j}^k := \left( (k-1) M + i \frac{M}{P}, j \frac{M}{P} \right),
\end{equation}
where $i,j$ lie in $[\![ 0,P ]\!]$. By some abuse of notation, (\ref{def-points}) 
will still be used for\hspace{0.01cm} 
 $i \!=\! P \!+\! 1$ (yet $x_{P+1,j}^k$ \hspace{0.01cm}
does not belong to $S_k^P$). Notice that $S_k^P$ also depends on $M$ and $N$, 
but this dependence (which will be clear in each context) is suppressed just to 
avoid overloading the notation. 

To simplify, we will only work with Delone subsets $\mathcal{D}$ of $\mathbb{Z}^2$ 
satisfying what we call the {\em $2\mathbb{Z}$-property}: all points $(m,n)$ with an 
even $m$ do belong to $\mathcal{D}$. In particular, we will consider domino tilings 
of the plane made only of the pieces 1-1 and 1-0. More generally, we say that a subset 
$\mathcal{D} \subset [\![A,B]\!] \times [\![A',B']\!]$ satisfies the {\em $2\mathbb{Z}$-property} if all 
points $(m,n) \in [\![A,B]\!] \times [\![A',B']\!]$ with an even $m$ do belong to $\mathcal{D}$.

There is a little technical problem that arises when considering maps defined 
on strict subsets of either $\mathbb{Z}^2$ or $R_{M,N}$. To overcome this, 
we introduce a general construction. Namely, given either a Delone set 
$\mD \subset \ZZ^2$ or a subset $\mD \subset R_{M,N}$ satisfying the 
$2\mathbb{Z}$-property in each case, for every function $f\!: \mD\to \ZZ^2$ 
we define its extension $\hat{f}$ to either $\ZZ^2$ or $R_{M,N}$ taking values 
in $\frac{1}{2} \ZZ^2$ by letting 
$$
\hat{f}(x)=\left \{ \begin{array}{ll}
                             f(x) &  \mbox{ if } x\in \mD,\\
                             f \big( x+(1,0) \big) - (\frac{1}{2},0) & \mbox{ if }  x\notin \mD.
                          \end{array}\right.
$$
The proof of the next lemma is straightforward and we leave it to the reader.

\vspace{0.2cm}

\begin{lem}\label{extension}
{\em If $f \!: \mD \to \ZZ^2$ is $L$-bi-Lipschitz, then $\hat{f}$ is a $6L$-bi-Lipschitz map.}
\end{lem}

\vspace{0.2cm}

The technical key of the construction is given by the next

\vspace{0.1cm}

\begin{lem}\label{existencia regular} {\em Given $L \geq 1$, a positive $\tau < 1$ and an integer 
$P \geq 1$, there exist $\lambda > 0$ and positive integers $M_0,N_0$ such that the following holds: 
Given a multiple $M \geq M_0$ of $P$ and a subset $\mD \subset R_{M,N}$, with $N \geq N_0$, 
satisfying the $2\mathbb{Z}$-property, let $f \!: \mD \to \mathbb{Z}^2$ be an $L$-bi-Lipschitz map, and 
denote $v_{M,N}^f := f(2MN,0) - f (0,0)$. Assume that for all points of the form $x_{i,j}^k$ above that do 
belong to $\mD$,}
\begin{equation}\label{kappa1}
\frac{\big\| f ({x}_{i+1,j}^{k}) - f (x_{i,j}^k) \big\|}{M/P}
\leq (1 + \lambda) \frac{\|v_{M,N}^f \|}{2MN}
\end{equation} 
\begin{equation}\label{kappa2}
\left( \mbox{resp. } \quad \frac{ \big\| f ({x}_{i+1,j}^{k} + (1,0)) - f (x_{i,j}^k) \big\| }{1+M/P}
\leq (1 + \lambda) \frac{\| v^f_{M,N}\|}{2MN} \right)\!,
\end{equation}
{\em provided $x_{i+1,j}^k$ (resp. $x_{i+1,j}^k + (1,0)$) lies in $\mD$.} 
{\em Then there is $k_* \!\in\! [\![ 1,2N-1 ]\!]$ such that for all $x_{i,j}^{k_*} \in S_{k_*}^P$,} 
\begin{equation}\label{rregular}
\frac{\big\langle \hat{f}(x_{i,j}^{k_*} + (M,0)) - \hat{f} (x_{i,j}^{k_*}) , v^f_{M,N} \big\rangle }{M} 
\geq (1 - \tau) \frac{\| v^f_{M,N} \|^2}{2MN}.
\end{equation}
\end{lem}

\begin{proof} We will deal with $\hat{f}$ instead of $f$. Accordingly, we denote $\hat{L} := 6L$.
Notice that in case $x_{i+1,j}^k$ does not belong to $\mD$, we still have 
\begin{small}
$$\frac{ \big\| \hat{f} ({x}_{i+1,j}^{k} \!-\! \hat{f} (x_{i,j}^k) \big\| }{M/P}
\leq
\frac{ \big\| f ({x}_{i+1,j}^{k} \!+\! (1,0)\!) \!-\! f (x_{i,j}^k) \big\| \!+\! \frac{1}{2}}{M/P}
\leq
\frac{(1 \!+\! \lambda) \| v^f_{M,N}\|}{2MN} \frac{(1\!+\!M\!/\!P)}{M/P} + \frac{P}{2M}.$$
\end{small}Thus,
\begin{equation}\label{nec}
\frac{ \big\| \hat{f} ({x}_{i+1,j}^{k} ) \!-\! \hat{f} (x_{i,j}^k) \big\| }{M/P}
\leq
(1 + 2\lambda) \frac{\| v^f_{M,N}\|}{2MN}, 
\end{equation}
where the last inequality holds provided 
$$\frac{P}{2M} \leq \frac{\lambda \| v^f_{M,N}\|}{4MN} \qquad \mbox{and} \qquad 
\left( 1+\frac{P}{M} \right) (1+\lambda) \leq 1 + \frac{3\lambda}{2},$$
which is always the case for $M \geq 
\max \big\{ \frac{LP}{\lambda}, \frac{2P(1+\lambda)}{\lambda} \big\}$.

Assume no square $S_{k*}^P$ satisfies the required property. A direct application 
of the pigeonhole principle then shows that there is a ``height'' $j_* \in [\![ 0, P ]\!]$ 
such that at least $r := \big[ \frac{2N-1}{2(P+1)} \big]$ squares $S_{k_1}, \ldots, S_{k_r}$ 
contain points $x_{i_1,j_*}^{k_1}, \ldots, x_{i_r,j_*}^{k_r}$, respectively, satisfying 
the reverse inequality to (\ref{rregular}) and such that all the indices $k_s$ 
have the same parity and are $\leq 2N-1$. 

Notice that $v_{M,N}^f$ equals 
$$\hat{f} (0,j_*) - \hat{f} (0,0) + \sum_{s=1}^{r} \big( \hat{f} (x_{i_s,j_*}^{k_s}) - \hat{f} (x_{i_s,j_*}^{1+k_s}) \big)
+ \hat{f} (2MN,0) - \hat{f} (2MN, j_*) + $$
$$+ \left[ \hat{f} (2MN,j_*) - \hat{f} (x_{i_r,j_*}^{k_r}) +  
 \sum_{s=2}^{r} \hat{f} (x_{i_s,j_*}^{k_s}) - \hat{f} (x_{i_{s-1},j_*}^{1+k_{s-1}}) + 
\hat{f} (x_{i_1,j_*}^{1+k_1}) - \hat{f} (0,j_*) \right].$$ 
The (non normalized) projections over $v_{M,N}^f$ of the expression into brackets can be estimated using the hypothesis: 
it is smaller than or equal to 
$$\left( 2MN - M \Big[ \frac{2N-1}{2(P+1)} \Big] \right) (1+2\lambda) \frac{\| v_{M,N}^f \|^2}{2MN}.$$ 
Therefore, by the choice of the points $x_{i_s,j_*}^{k_s}$, the value of   
$\| v_{M,N}^f \|^2$ is bounded from above by 
$$\big\langle \hat{f}(0,j_*)-\hat{f}(0,0) , v^f_{M,N} \big\rangle 
+ (1-\tau)\frac{2N - 1}{2(P+1)}\frac{\| v^f_{M,N} \|^2}{2N} 
+ \big\langle \hat{f} (2MN,0) - \hat{f} (2MN, j_*), v_{M,N}^f \big\rangle +$$
$$+ \left( 2MN - M \Big[ \frac{2N-1}{2(P+1)} \Big] \right) (1+2\lambda) \frac{\| v_{M,N}^f \|^2}{2MN}.$$
Since $\hat{f}$ is $\hat{L}$-Lipschitz, we finally conclude that  
\begin{eqnarray*}\label{regular}
\| v_{M,N}^f \|^2 \leq 2\hat{L}M \|v^f_{M,N}\| + \left ( \! 1-\frac{1}{2N} 
\Big[\frac{2N-1}{2(P+1)} \Big] \! \right ) \! (1+2\lambda) \| v^f_{M,N} \|^2 
+(1-\tau)\frac{2N - 1}{2(P+1)}\frac{\| v^f_{M,N} \|^2}{2N}.
\end{eqnarray*}
Thus we get
$$
\| v^f_{M,N} \| \left( -2\lambda+\frac{(1+2\lambda)}{2N} \Big[ \frac{2N-1}{2(P+1)} \Big] -  
\frac{(1-\tau)}{2N} \frac{2N-1}{2(P+1)} \right) < 2\hat{L}M,
$$
hence
$$
\| v^f_{M,N} \| \left( -2\lambda+\frac{(1+2\lambda)(2N-2P-2)}{4N(P+1)}  -  
\frac{1-\tau}{2(P+1)} \right) < 2\hat{L}M.  
$$
For $N>\frac{2(P+1)}{\tau}$, we have 
$$\frac{N \tau - P - 1}{2 (2NP + N + P + 1)} > \frac{\tau}{12P},$$ 
thus for $\lambda \leq \frac{\tau}{12P}$, we obtain
$$-2\lambda+\frac{(1+2\lambda)(2N-2P-2)}{4N(P+1)}  -  \frac{1-\tau}{2(P+1)}  > 0.$$
The bi-Lipschitz condition of $f$ then yields 
$$
\frac{2NM}{L} \left( -2\lambda + \frac{(1+2\lambda)(2N-2P-2)}{4N(P+1)} - \frac{1-\tau}{2(P+1)} \right) 
< 2\hat{L}M.
$$
However, one easily checks that given $M$, this is impossible for 
$$ \lambda \leq \frac{\tau}{12P}, 
\qquad N \geq N_0 := 1 + \left[ \frac{1}{\tau} \big( L\hat{L} + 2P + 2 + \tau \big) \right].$$
This finishes the proof for $M \geq M_0 := \frac{(L+4)P}{\lambda}$.
\end{proof}

\vspace{0.6cm}

In analogy to the terminology introduced in \cite{BK1}, a square 
$S_{k_*}^P$ satisfying the conclusion of the preceding lemma ({\em i.e.} 
condition (\ref{rregular})) will be said to be $(M, N, \tau, f)$-regular.
 
\vspace{0.4cm}

\begin{lem}\label{distance}
{\em Given $L\geq 1$, $\varepsilon>0$ and an integer $P \geq 1$, there exists a positive 
$\tau < 1$ such that the following holds: Let $M \geq M_0:= \frac{(L+4)P}{\lambda}$ be a 
multiple of $P$, where $\lambda := \frac{\tau}{12P}$. Suppose $f  \!\!: \mD \to \mathbb{Z}^2$ 
is an $L$-bi-Lipschitz map such that for each $x_{i,j}^k \in S_k^P$, either (\ref{kappa1}) or 
(\ref{kappa2}) holds according to the case. Then for every $x_{i,j}^{k_*}$ belonging to an  
$(M,N,\tau, f)$-regular square $S_{k_*}^P$, one has}   
\begin{equation}\label{eq-distance}
\left \| \frac{\hat{f}(x^{k_*+1}_{i,j})-\hat{f}(x_{i,j}^{k_*})}{M}-\frac{v_{M,N}^f}{2MN}\right\|\leq \varepsilon.
\end{equation}
\end{lem}

\begin{proof}
Again, we denote $\hat{L}:=6L$. Given $x_{i,j}^{k_*} \in S^P_{k_*}$, let us write 
$$\hat{f} (x_{i,j}^{k_*+1}) - \hat{f} (x_{i,j}^{k_*}) = \alpha_{i,j}^{k_*} v_{M,N}^f 
+ \beta_{i,j}^{k_*} v_{M,N}^{\perp}$$ 
for certain reals $\alpha_{i,j}^{k_*}$ and $\beta_{i,j}^{k_*}$, where $v_{M,N}^{\perp}$ 
is a unit vector orthogonal to $v_{M,N}^f$. On the one hand, by (\ref{rregular}), 
$$\alpha_{i,j}^{k_*} \frac{\| v_{M,N}^f\|^2}{M} \geq (1-\tau) \frac{\| v_{M,N}^f \|^2}{2MN},$$
hence
\begin{equation}\label{lower}
\alpha_{i,j}^{k_*} \geq \frac{1 - \tau}{2N}.
\end{equation}
On the other hand, since $M \geq M_0$, using $P$ times (\ref{nec}) and the triangle inequality, 
we obtain
$$(\alpha_{i,j}^{k_*})^2 \| v_{M,N}^f \|^2
\leq (\alpha_{i,j}^{k_*})^2\|v_{M,N}^f\|^2 + (\beta_{i,j}^{k_*})^2 
\leq \left( (1 + 2 \lambda) \frac{\| v_{M,N}^f \|}{2N} \right)^2.$$
Therefore, 
$$|\alpha_{i,j}^{k_*}| \leq \frac{1 + 2\lambda}{2N} \leq \frac{1+\tau}{2N}.$$ 
Similarly, using (\ref{kappa1}), (\ref{lower}) and the previous estimate, we obtain 
$$\left( \frac{(1-\tau)^2 \| v_{M,N}^f \|^2}{4N^2} + (\beta_{i,j}^{k_*})^2 \right) 
\leq (\alpha_{i,j}^{k_*})^2\|v_{M,N}^f\|^2 + (\beta_{i,j}^{k_*})^2
\leq (1 + 2\lambda)^2 \frac{\| v_{M,N}^f \|^2}{4N^2},$$
which yields 
$$(\beta_{i,j}^{k_*})^2 \leq \frac{\|v_{M,N}^f\|^2}{4N^2} 
\left( (1+\tau)^2 - (1-\tau)^2 \right) 
= \frac{\tau \|v_{M,N}^f\|^2}{N^2}.$$
As a consequence, 
\begin{small}
\begin{eqnarray*}
\left\| \frac{\hat{f}(x^{k+1}_{i,j}) \!-\! \hat{f}(x_{i,j}^k)}{M}-\frac{v_{M,N}^f}{2MN} \right\| 
\!\!\!
&=& \!\!\!
\left\| \frac{\alpha_{i,j}^k}{M} v_{M,N}^f 
      + \frac{\beta_{i,j}^k}{M} v_{M,N}^{\perp} - \frac{v_{M,N}^f}{2MN} \right\| 
\leq \| v_{M,N}^f  \|\left| \frac{\alpha_{i,j}^k}{M} -\frac{1}{2MN}  \right| + \frac{| \beta_{i,j}^k |}{M}\\
&\leq& \!\!\!  \| v_{M,N}^f  \| \frac{\tau}{2MN} + \|v_{M,N}^f\|\frac{\sqrt{\tau}}{M N}
\leq 2MNL\left(  \! \frac{\tau}{2MN} \!+\! \frac{\sqrt{\tau}}{M N}  \! \right)
\leq  3 L \sqrt{\tau}
= \varepsilon,
\end{eqnarray*}
\end{small}where the last equality holds for $\tau :=\frac{\varepsilon^2}{9L^2}$.
\end{proof}

\vspace{0.2cm}

Below we put together the two preceding lemmas into a single statement.

\vspace{0.1cm}

\begin{prop} \label{prop-regularidad}
{\em Given $L \geq 1$, a positive $\varepsilon < 1$ and a positive integer $P$, 
there exist $\lambda > 0$ and positive integers $M_0, N_0$ such that the following 
holds: Given a subset $\mD \subset R_{M,N}$ satisfying the $2\mathbb{Z}$-property, 
with $M \geq M_0$ a multiple of $P$ and $N \geq N_0$, let $f \!: \mD \to \mathbb{Z}^2$ 
be an $L$-bi-Lipschitz map. Assume that for every point of the form $x_{i,j}^k$ 
that belongs to $\mD$,}  
$$\frac{\big\| f ({x}_{i+1,j}^{k}) - f (x_{i,j}^k) \big\|}{M/P}
\leq (1 + \lambda) \frac{\|v_{M,N}^f \|}{2MN}$$
$$\left( \mbox{resp. } \quad \frac{ \big\| f ({x}_{i+1,j}^{k} + (1,0)) - f (x_{i,j}^k) \big\| }{1+M/P}
\leq (1 + \lambda) \frac{\| v^f_{M,N}\|}{2MN} \right)\!,$$
{\em provided $x_{i+1,j}^k$ (resp. $x_{i+1,j}^k + (1,0)$) lies in $\mD$.} 
{\em Then there is a subset 
$$S = S_k^P := \left\{ \Big( (k-1)M + i \frac{M}{P} , j \frac{M}{P} \Big) \!: 
0 \leq i \leq P, \hspace{0.1cm} 0 \leq j \leq P \right\}$$ 
such that every $x \in S$ satisfies}   
\begin{eqnarray*}\label{new-epsilon}
\left \| \frac{\hat{f} ( x+(M,0) )-\hat{f}(x)}{M}-\frac{ f(2M N,0) - f(0,0) }{2MN}\right\| 
\leq \varepsilon.
\end{eqnarray*}
\end{prop}

\vspace{2mm}

\begin{rem} \label{primera-estimacion}
The estimates and definitions given along the proofs of Lemmas \ref{existencia regular} 
and \ref{distance} show that, given $L \geq 1$, a positive constant $\varepsilon < 1$ and 
a positive integer $P$, the conclusion of Proposition \ref{prop-regularidad} holds for 
$$\lambda \leq \frac{\varepsilon^2}{108 P L^2},$$
$$M_0 \geq \frac{108 P^2 L^2 (L+4)}{\varepsilon^2},$$
and
$$N_0 \geq 2 + \frac{216 L^2 P (3 L^2 + P + 1)}{\varepsilon^2}.$$ 
\end{rem}


\noindent{\bf II. {\em Coarse differentiability forces densities to be close.}} Let 
$f\!: \mD \to \ZZ^2$ be an $L$-bi-Lipschitz map defined on a Delone set $\mD \subset \ZZ^2$ 
satisfying the $2\mathbb{Z}$-property. Fix an integer $P \geq 1$, and let $S$ be a square of the 
form $S_k^P \subset R_{M,N}$, where $M$ is a multiple of $P$. We let $\gamma^*$ be the 
curve obtained by connecting (using line segments) points in $\hat{f}(\partial S)$ 
coming from consecutive points 
in $\partial S$. The curve $\gamma^*$ is closed though not necessarily simple. However, it contains the simple 
curve $\gamma = \gamma_S$ obtained by ``deleting short loops''. Notice that the bi-Lipschitz property of $\hat{f}$ 
easily implies that each loop has length at most $2 \hat{L}^3 M/P$. Therefore, if $P \geq 4 \hat{L}^4$, then 
$\gamma$ has length at least 
$$\frac{M\sqrt{2}}{\hat{L}} - \frac{4 \hat{L}^3 M}{P} \geq 4 (\sqrt{2}-1) \hat{L}^3 > 0.$$
In particular, it is well defined. We denote by $int(\gamma)$ (resp. $ext(\gamma)$) 
the closed, bounded (resp. unbounded) region of the plane determined by $\gamma$.  

We let 
$$\hat{S} :=  \big\{ \big( (k-1)M + i, j \big) \!: i,j \mbox{ in } [\![ 0, M-1 ]\!] \big\}.$$
This corresponds to the set of all points with integer coordinates in the region (square) 
bounded by the points of $S_k^P$, except for those in the upper and the right sides of 
the square. We call such a subset the {\em lower-left corner} of the corresponding square.

Given $\varepsilon > -1$, we let $(1+\varepsilon) \hat{S}$ be the set of all 
points with integer coordinates lying in the square having the same center as 
$S$ though side of length $(1+\varepsilon)M$. We also denote by $\mathcal{S}_1$ 
the unit square in $\mathbb{R}^2$, and by $(1+\varepsilon)\mathcal{S}_1$ the 
corresponding homothetic copy. 

\vspace{0.2cm}

\begin{lem} \label{int-ext}
{\em Given $L \geq 1$ and $\varepsilon > 0$, there exists $P_0$ such that the 
following holds: If $f : \mD \to \ZZ^2$ is $L$-bi-Lipschitz and $P \geq P_0$, then

\vspace{0.1cm}

\noindent $\mathrm{(i)}$ no point of $\hat{f} \big( \ZZ^2 \setminus (1+\varepsilon) \hat{S}) \big)$ 
lies in $int(\gamma)$;

\vspace{0.1cm}

\noindent $\mathrm{(ii)}$ all points in $\hat{f} \big( (1-\varepsilon)\hat{S} \big)$ are contained 
in $int(\gamma)$.}
\end{lem}

\vspace{0.1cm}

This lemma can be easily shown by contradiction just by renormalising 
and passing to the limit (along a subsequence) using a variation of the 
Arzela-Ascoli theorem. Indeed, such an argument provides a limit 
homeomorphism $F$ from the unit square $\mathcal{S}_1$ as well as:

\noindent -- In case (i), a point in the exterior of $(1+\varepsilon)\mathcal{S}_1$ 
which is mapped by $F$ inside $F(\mathcal{S}_1)$;

\noindent -- In case (ii), a point in $(1-\varepsilon) \mathcal{S}_1$ which is mapped 
by $F$ into a point outside $F(\mathcal{S}_1)$.

\noindent In each case, this is certainly impossible, since $F$ is an homeomorphism. 

Despite this simple argument, it is better to give a slightly more involved proof 
that yields a quantitative estimate for $P_0$ in terms of $L$ and $\varepsilon$. 

\vspace{0.4cm}

\noindent{\em Proof of Lemma \ref{int-ext}.} We claim that the lemma holds for 
$P_0 :=\max \big\{ 4 \hat{L}^4, 3 \hat{L}^2 / \varepsilon \big\}$. 

For (i), let $x \in \ZZ^2 \setminus \hat{S}$ be a point that is mapped by $\hat{f}$ inside 
$int(\gamma)$ and lies at a maximal distance of $\hat{S}$ among these points. (Notice that,  
by the bi-Lipschitz property and the Delone condition, only finitely many points map into $int(\gamma)$.) 
We claim that $dist (\hat{f}(x), \gamma) \leq \hat{L}$. Otherwise, the 
closed ball of center $\hat{f}(x)$ and radius $\hat{L}$ would be contained in $int(\gamma)$. 
This ball contains the image under $\hat{f}$ of the points $x - (1,0), x+ (1,0), x - (0,1), x+ (0,1)$. 
However, among these points, at least one lies at distance of $\hat{S}$ strictly larger than that of 
$x$, which contradicts the choice of $x$. 

Now, it is obvious from the construction that every point in $\gamma$ lies at distance 
$\leq \frac{\hat{L} M}{P}$ from some point of the form $\hat{f}(y)$, where $y \in \partial S$.  
Therefore, 
$$\frac{\| x-y \|}{\hat{L}} 
\leq 
\big\| \hat{f}(x) - \hat{f}(y) \big\| 
\leq 
\hat{L} + \frac{\hat{L}M}{P} = 
\hat{L} \left( 1 + \frac{M}{P} \right),$$ 
hence
$$dist (x,\partial \hat{S}) 
\leq 
1 + \| x-y \| 
\leq 
\hat{L}^2 \left( 2 + \frac{M}{P} \right) \leq \frac{3 \hat{L}^2 M}{P} \leq \varepsilon M,$$
where the last inequality holds provided $P \geq P_0$.

The proof of (ii) proceeds analogously dealing with $f^{-1}$ instead of $f$. 
$\hfill\square$

\vspace{0.25cm}

\begin{rem} It is an open problem whether every bi-Lipschitz map defined on a Delone subset of the plane 
can be extended into a bi-Lipschitz homeomorphism of the whole plane (see \cite[Question 4.14.(ii)]{alestalo}). 
Certainly, having an affirmative answer for (a quantitative version of) this question would yield another 
proof of the preceding lemma. The estimates given above are, however, enough for our purposes.
\end{rem}

\vspace{0.2cm}

The next elementary lemma will be needed when comparing cardinalities of points 
enclosed by curves each of which is an almost translated copy of the other one. 

\begin{lem} \label{isoper}
{\em If $\gamma$ is a rectifiable curve in $\RR^2$ of $\mathrm{length}(\gamma) \geq 4$ and 
$1 \leq T \leq \frac{\mathrm{length}(\gamma)}{4}$, then}
$$
\big| \{x\in \ZZ^2: d(x,\gamma)\leq T \} \big| \leq 25 \hspace{0.1cm} T \hspace{0.1cm} \mbox{length} (\gamma).
$$
\end{lem}

\begin{proof}
Let $x_1,\ldots, x_k$ be points in $\gamma$ such that every $x\in \gamma$ has distance $\leq T$ 
to at least one of the points $x_i$.  Notice that we can take such a $k\in \mathbb{N}$ satisfying 
$$k \leq 2 + \frac{\mbox{length}(\gamma)}{2T} \leq \frac{\mbox{length}(\gamma)}{T}.
$$
If $x\in \ZZ^2$ satisfies $d(x,\gamma)\leq T$, then $\|x-x_i\|\leq 2T$ holds for some 
$1\leq i\leq k$. Therefore, 
$$
\big\{ x\in\ZZ^2: d(x,\gamma)\leq T \big\}
\subseteq 
\bigcup_{i=1}^k \big\{ x\in \ZZ^2: \|x-x_i \|\leq 2T \big\}. 
$$
Thus,
\begin{eqnarray*}
\big| \{x\in\ZZ^2: d(x,\gamma)\leq T\} \big| 
\leq k (4T+1)^2
\leq \frac{\mbox{length}(\gamma)}{T} (5T)^2 
= 25 \hspace{0.1cm} T \hspace{0.1cm} \mbox{length} (\gamma),
\end{eqnarray*}
which finishes the proof.
\end{proof}

We can now state and prove the main argument involving local densities of points of $\mD$ 
via comparison along the images.

\begin{prop} \label{densidad}
{\em Given $L \geq 1$ and $1 \geq d > d' > 0$, there exist a positive $\varepsilon < 1$ and integers 
$P_1,M_1$ such that the following holds: Let $\mD$ be a Delone set satisfying the $2\mathbb{Z}$-property, 
and let $f \!: \mD \to \mathbb{Z}^2$ be $L$-bi-Lipschitz and 
surjective. Assume that for $P \geq P_1$, $N \geq 1$ and 
$M \geq M_1$ a multiple of $P$, some square $S := S_k^P \subset R_{M,N}$, with 
$1 \leq k < 2M$ is such that every 
$x \in S$ satisfies (\ref{eq-distance}), and denote $S' := S_{k+1}^P$. If $\hat{S} \cap \mD$ contains 
$\geq d M^2$ (resp. $\leq d' M^2$) points and $\hat{S}' \cap \mD$ contains $\leq d' M^2$ (resp. 
$\geq dM^2$ points), then $f$ cannot be $L$-bi-Lipschitz.}
\end{prop}

\begin{proof} We will show that the claim holds for all $\varepsilon \!< \! \frac{d-d'}{20(2+5L)}$,  
$M_1 \geq \max \{2\hat{L}, 1/\varepsilon\}$ and $P_1=P_0$, where $P_0$ is given by 
Lemma \ref{int-ext}. To do this, we will suppose that $|\hat{S} \cap \mD| \geq dM^2$ 
and $|\hat{S}' \cap \mD| \leq d'M^2$, the other case being analogous.

We proceed by contradiction. Assuming that $f$ is $L$-bi-Lipschitz, we use Lemma \ref{int-ext}. 
By (ii), for $\gamma := \gamma_S$, the set $f(\mD \cap (1-\varepsilon)\hat{S}) \subset \ZZ^2$ contains  
$\geq d M^2 - 4 (\varepsilon M + 1)^2$, all lying in $int(\gamma)$:
$$\big| int(\gamma) \cap \ZZ^2\big| \geq dM^2 - 16\varepsilon M^2.$$
By (i) and the surjectivity of $f \! : \mathcal{D} \to \mathbb{Z}$, 
for $\gamma' := \gamma_{S'}$, the set $int(\gamma') \cap \ZZ^2$ is contained in 
$f( \mD \cap (1+\varepsilon)\hat{S}')$, hence its cardinality is bounded from above by 
$d'M^2 + 4 \varepsilon M (M+1) + 4(\varepsilon M + 1)^2$:
$$\big| int(\gamma') \cap \ZZ^2 \big| \leq d'M^2 + 24 \varepsilon M^2.$$

We claim that points of $int (\gamma)$ must lie in $int(\gamma')$ after translation 
by $\frac{v_{M,N}}{2N}$, except perhaps for those which are moved into points that 
are $\varepsilon M$-close to $\gamma'$. Indeed, $\gamma$ (hence $int(\gamma)$) is 
determined by the image $\hat{f} (\partial S)$, hence by points of the form $\hat{f} (x_{i,j}^k)$ 
for which (\ref{eq-distance}) holds. Obviously, similar arguments apply to $\gamma'$. 

We next claim that we may use the preceding lemma to conclude that the number 
of points that move into points $\varepsilon M$-close to $\gamma'$ is at most 
$$ 25 \varepsilon M \mbox{length}(\gamma') \leq 100 \varepsilon LM^2.$$ 
Indeed, the choices of $P$ and $M$ yield  
$$\mbox{length}(\gamma') \geq 2M/\hat{L}  > 4 \quad \mbox{ and } 
\quad \varepsilon M \leq 1 /4 \leq \mbox{length}(\gamma')/4,$$ 
thus fulfilling the hypothesis of Lemma \ref{isoper}. 

The preceding estimates force  
$$d M^2 - 16\varepsilon M^2 - 100 \varepsilon LM^2 \leq d' M^2 + 24\varepsilon M^2,$$
that is, 
$$d \leq d' + (40 + 100L) \varepsilon.$$
However, this is impossible due to the choice of $\varepsilon$. 
\end{proof}

\vspace{0.3cm}

We next put together Propositions \ref{prop-regularidad} 
and \ref{densidad} into a single one.

\vspace{0.1cm} 

\begin{prop} \label{expandiendo}
{\em Given $L \geq 1$ and $1 \geq d > d' > 0$, there exist $\lambda > 0$ 
and positive integers $M_*,N_*,P_*$ such that the following holds: Let 
$\mD$ be a Delone set satisfying the $2\mathbb{Z}$-property, and let 
$f \!: \mD \to \mathbb{Z}^2$ be $L$-bi-Lipschitz and surjective. Assume 
that for $M \geq M_*$ and $N \geq N_*$, with $M$ a multiple of $P_*$, there 
are two consecutive squares $S_k^{P_*},S_{k+1}^{P_*}$ of $R_{M,N}$ such 
that the lower-left corner of one of them contains at least $dM^2$ points of 
$\mD$, and the lower-left corner of the other one has no more than 
$d'M^2$ points of $\mD$. Then there must exist a point 
$x \in \mD \cap R_{M,N}$ of the form $x_{i,j}^k$ such that either}
$$\frac{\big\|  f(x+(M/P,0)) - f(x) \big\|}{M/P} 
\geq (1 + \lambda) \frac{\big\| f(2MN,0) - f(0,0) \big\|}{2MN}$$
if $x + (M/P,0)$ belongs to $\mD$, or 
$$\frac{\big\| f(x+(1+M/P,0)) - f(x) \big\|}{1+M/P} 
\geq (1 + \lambda) \frac{\big\| f(2MN,0) - f(0,0) \big\|}{2MN}$$
otherwise.
\end{prop}

Roughly, the preceding Proposition says that if a Delone set $\mD$ with the $2\mathbb{Z}$-property 
maps onto $\mathbb{Z}^2$ by an $L$-bi-Lipschitz map $f$, then variations of the local density of 
$\mD$ force the Lipschitz constant of $f$ to increase when passing from a certain scale to a 
smaller one. By inductive application of this argument, we will contradict the Lipschitz 
condition of $f$ for appropriately constructed Delone sets.

\vsp

\begin{rem} \label{segunda-estimacion}
The estimates of Remark \ref{primera-estimacion} and those given in Lemma \ref{int-ext} and 
Proposition \ref{densidad} show that, given $L \geq 1$ and $1 \geq d > d' > 0$, the conclusion 
of Proposition \ref{expandiendo} holds for 
$$\lambda \leq \frac{(d-d')^3}{10^{10} L^7}, \qquad 
M_0 \geq \frac{10^{15} L^{11}}{(d-d')^4}, \qquad 
N_0 \geq \frac{10^{10} L^{10}}{(d-d')^4}.$$
\end{rem}

\vspace{0.35cm}


\noindent{\bf III. {\em Construction of the non-rectifiable, repetitive Delone set.}}  
We start by introducing a general recipe for constructing repetitive Delone subsets of $\ZZ^2$. 

\vspace{0.1cm}

Let $(F_n)_{n\geq 1}$ be a sequence of finite subsets of $\ZZ^2$ satisfying the following properties:
\begin{itemize}
\item[(\bf{F1})] $(0,0) \in F_n\subseteq F_{n+1}$, for every $n\geq 1$;

\item[\bf{(F2)}]  $\ZZ^2=\bigcup_{n\geq 1} F_n$;

\item[\bf{(F3)}] For every $n\geq 1$,  the set $F_{n+1}$ is a disjoint union of translated copies of $F_n$.   
\end{itemize}
The last condition yields a finite subset $\Gamma_n \subset F_{n+1}$ such that 
$$
F_{n+1}=\bigcup_{v\in \Gamma_n} (F_n+v).
$$ 

Assume that for each $n \geq1$, there exist $k_n \geq 1$ and a family of {\em patches} 
$\mP_{n,1},\ldots, \mP_{n,k_n}$ in $\{ 0,1 \}^{F_n}$ such that: 
\begin{itemize}
\item[\bf{(F4)}] $\mP_{n+1,k}|_{v+F_n} - v$ belongs to $\{\mP_{n,1},\ldots, \mP_{n,k_n}\}$ 
for all $v \in \Gamma_n$ and all $k \in [\![ 1,k_{n+1} ]\!]$;

\item[\bf{(F5)}]  For every $j \in [\![ 1, k_n ]\!]$ and $k \in [\![ 1, k_{n+1} ]\!]$, 
one has $\mP_{n+1,k}|_{v + F_n} = \mP_{n,j}$ for a certain $v\in \Gamma_n$;

\item[\bf{(F6)}]  $\mP_{n+1,1}|_{F_n} = \mP_{n,1}$.
\end{itemize}
By properties (F1), (F2) and (F6) above, the intersection 
$$
\bigcap_{n\geq 1} \big\{ D \in \{0,1\}^{\ZZ^2} \!: D|_{F_n} = \mP_{n,1} \big\}
$$ 
consists of a single point, which can be viewed as a subset $\mD$ of $\ZZ^2$.

\vspace{0.1cm}

\begin{lem}\label{lema-repetitivo}
The set $\mD$ is a repetitive Delone set.
\end{lem}

\begin{proof}
Fix $r>0$. Since $\mD$ is a subset of $\ZZ^2$, only 
finitely many patches $\mathcal{Q}_1,\ldots, \mathcal{Q}_m$ of diameter 
$2r$ appear (up to translation) in $\mD$.  Let $n \geq 1$ be such that the restriction 
of $\mD$ to $F_n$ ({\em i.e.} $\mP_{n,1}$) contains (translated copies of) all of the patches 
$\mathcal{Q}_1,\ldots, \mathcal{Q}_m$. Property (F5) above ensures that for a large-enough 
$R>0$, every ball of radius $R$ in $\mD$ cointains a translated copy of the patch $\mP_{n,1}$, 
hence a copy of each patch $\mathcal{Q}_1,\ldots,\mathcal{Q}_m$. Thus, every ball of 
radius $r$ appears in each ball of radius $R$. 
\end{proof}

\vspace{0.2cm}

In order to implement the strategy above, we need to specify our building blocks 
({\em i.e.} the patches along the construction). These will be constructed 
starting from two data, namely:

\begin{itemize}

\item A constant $L \geq 1$ (which will play the role of the Lipschitz constant to discard);

\item Two square patches $\mQ_1, \mQ_2$ in $\mathbb{Z}^2$ 
that have equal and even length-side but contain different number of points. We let 
$d_i$ be the density of points in the lower-left corner of $\mQ_i$, the notation being 
such that $d_2 > d_1$. We also assume that both patches contain all boundary points 
and satisfy the $2\mathbb{Z}$-property when placed centered at the origin. 

\end{itemize}
 
\noindent Given these data, fix $d_1',d_2'$ such that $d_2 > d_2' > d_1' > d_1$.  Let 
$\lambda, M_*,N_*,P_*$ be the constants provided by Proposition \ref{expandiendo} 
for $L$, $d := d_2'$ and $d' := d_1'$. Fix an integer $\ell \geq 1$ such that 
\begin{equation}\label{condicion-kappa}
\frac{(1+\lambda)^{\ell}}{L} > L.
\end{equation}
Using the elementary inequality \hspace{0.1cm} 
$(1 + \lambda)^{\ell} \geq 1 + \lambda \ell$, \hspace{0.01cm} 
one easily checks that this holds for
\begin{equation}
\label{condition-kappa-2}
\ell \geq \frac{L^2}{\lambda}.
\end{equation}

Let $2M$ be the side-length of the patches $\mQ_i^{0} := \mQ_i$, $i \in \{1,2\}$.  We view these 
patches as subsets of $[\![ -M,M ]\!] \times [\![ -M,M ]\!]$, that is, centered at the origin. We start 
by constructing new patches $\mQ_1^{1},\mQ_2^{1}$ as follows (see Figure 1):

\begin{itemize}

\item Fix an odd positive integer $m$ 
so that $2m P_* M \geq M_*$, and form a square (centered 
at the origin) of $(mP_*)^2$ copies of $\mQ_1$ matching 
left sides to right sides and lower sides to upper sides. 

\item Next, match to the right a square block consisting of $(mP_*)^2$ copies of $\mQ_2$. 
After this, match to the right a square block consisting of $(mP_*)^2$ copies of $\mQ_1$. Proceed 
similarly up to having matched $N$ blocks made of pieces $\mQ_1$ and $\mQ_2$ in an alternate 
way, where the integer $N \geq N_*$ is to be fixed below. 

\item Proceed similarly to the left of the centered-at-the-origin block made of pieces $\mQ_1$. 
In this way, we form a rectangle of sides \hspace{0.01cm} $2mP_*M(2N+1)$ 
\hspace{0.01cm} and \hspace{0.01cm} $2mP_*M$, \hspace{0.01cm} filled 
by alternate blocks of copies of $Q_1$ and $Q_2$.

\item To complete $\mQ_{1}^1$, fill up the whole square of side $2mP_*M(2N+1)$ 
centered at the origin by matching copies of $\mQ_1$ at all places, except for 
those in the lower rectangle of sides \hspace{0.01cm} $2mP_*M(2N+1)$ 
\hspace{0.01cm} and \hspace{0.01cm} $2mP_*M$, \hspace{0.01cm} 
where we match the rectangle constructed above.  (We emphasize that all matchings 
are made as above, that is, by identifying left to right sides, and lower to upper sides).

\item Finally, to construct $\mQ_2^1$, proceed similarly as for $\mQ_1^1$ 
switching the roles of $\mQ_1$ and $\mQ_2$. 
 
\item The integer $N$ is taken $\geq N_*$ and such that the density of points 
in the lower-left corner of $\mQ_1^1$ (resp. $\mQ_2^1$) is $< d_1'$ (resp. 
$>d_2'$). One can easily check that this holds for  $N$ satisfying 
\begin{equation}\label{condition-kappa-3}
N \geq 2 \max \left\{ \frac{N_*}{2}, \frac{1}{d_2-d_2'}, \frac{1}{d_1'-d_1} \right\}.
\end{equation}

\end{itemize}

Next, we repeat the procedure, but starting with the patches $\mQ_1^{1}, \mQ_2^{1}$, keeping 
the same constants $L, d_1', d_2'$. We thus get new patches $\mQ_1^2,\mQ_2^2$ of densities 
$< d_1'$ and $> d_2'$, respectively, to which we may apply the construction again... 
If we repeat this procedure $\ell$ times, we obtain new patches, that we  
denote by $\mQ_1^{new}$ and $\mQ_2^{new}$ (and that have densities 
$< d_1'$ and $> d_2'$, respectively). 

\vspace{0.1cm}

\begin{lem} \label{no-bilipschitz}
{\em Let $\mD$ be a Delone subset of $\ZZ^2$ satisfying the $2\mathbb{Z}$-property. 
If $\mD$ contains translated copies of either $\mQ_1^{new}$ or $\mQ_2^{new}$ as building blocks 
as above, then $\mD$ cannot be mapped onto $\ZZ^2$ by an $L$-bi-Lipschitz map. }
\end{lem}

\begin{proof} We call {\em expansion of points $x,y$ under a map $f$} the expression
$$\frac{\big\| f(x)- f(y) \big\|}{\big\| x-y \big\|}.$$ 
By Proposition \ref{expandiendo},  if $f$ is an $L$-bi-Lipschitz surjective map 
$\mD \to \mathbb{Z}^2$, the expansion of the end-points of the lower side of 
$\mQ^{\ell}_i$ is at most $\frac{1}{1+\lambda}$ times  the expansion of the end-points 
of the lower side of some square made of $mP_*$ copies of $\mQ_j^{\ell-1}$, where 
$m = m_{\ell}$. By the triangle inequality, the latter is larger than or equal to 
the expansion of the end-points of some of the patches $\mQ_j^{\ell-1}$ 
placed at the lower side of this square. 

By the construction, the preceding argument yields that the expansion above is no more than 
$\frac{1}{1+\lambda}$ times the expansion of the end-points of the lower side of a certain square 
$\mQ_{j'}^{\ell-2}$. Continuing this way, in $\ell$ steps, we get two pairs of points such that 
the expansion for one pair is at least $(1+\lambda)^{\ell}$ times that of the other pair. Now, as $f$ is 
$L$-bi-Lipschitz, both expansions are $\leq L$ and $\geq 1/L$. This is in contradiction to (\ref{condicion-kappa}).
\end{proof}

\vspace{0.2cm}

\begin{center}
 \resizebox{17cm}{!}{%
\begin{tikzpicture}
\draw(-18, -19.5)--(-17,-19.5);\draw(-18,-19.5) node{$|$};\draw(-17,-19.5) node{$|$}; \draw(-17.5,-20) node{$2M$};

\draw(-18, -18.5)--(-15,-18.5); \draw(-15,-18.5) node{$|$};\draw(-18,-18.5) node{$|$}; \draw(-16.5,-19) node{$2mP_*M$};

 \draw (-18,22) node{$|$}; \draw (-18,22) -- (-1,22); \draw (1,22) node{$2mP_*M(2N+1)$}; \draw (3,22) -- (21,22); \draw (21,22) node{$|$};

\draw (-18,-18) rectangle (21,21);
\draw[fill=gray!50] (-15,-18) rectangle (-12,-15); \draw[fill=gray!50] (-3,-18) rectangle (0,-15); \draw[fill=gray!50] (3,-18) rectangle (6,-15); \draw[fill=gray!50] (15,-18) rectangle (18,-15);

  \draw (-18,0) grid (-15,3);
 \draw (-17.5,0.5) node{$\mQ_1$}; \draw (-16.5,0.5) node{$\cdots$};  \draw (-15.5,0.5) node{$\mQ_1$};
 \draw (-17.5, 1.6) node{$\vdots$};  \draw (-16.5, 1.5) node{$\cdots $};  \draw (-15.5, 1.6) node{$\vdots$};
  \draw (-17.5,2.5) node{$\mQ_1$}; \draw (-16.5,2.5) node{$\cdots$};  \draw (-15.5,2.5) node{$\mQ_1$};

    \draw (-15,0) grid (-12,3);
 \draw (-14.5,0.5) node{$\mQ_1$}; \draw (-13.5,0.5) node{$\cdots$};  \draw (-12.5,0.5) node{$\mQ_1$};
 \draw (-14.5, 1.6) node{$\vdots$};  \draw (-13.5, 1.5) node{$\cdots $};  \draw (-12.5, 1.6) node{$\vdots$};
  \draw (-14.5,2.5) node{$\mQ_1$}; \draw (-13.5,2.5) node{$\cdots$};  \draw (-12.5,2.5) node{$\mQ_1$};

  \draw[dashed](-11,1.5)--(-7,1.5);

\draw (-6,0) grid (-3,3);
 \draw (-5.5,0.5) node{$\mQ_1$}; \draw (-4.5,0.5) node{$\cdots$};  \draw (-3.5,0.5) node {$\mQ_1$};
 \draw (-5.5, 1.6) node{$\vdots$};  \draw (-4.5, 1.5) node {$\cdots $};  \draw (-3.5, 1.6) node {$\vdots$};
  \draw (-5.5,2.5) node{$\mQ_1$}; \draw (-4.5,2.5) node{$\cdots$};  \draw (-3.5,2.5) node {$\mQ_1$};

\draw (-3,0) grid (0,3);
 \draw (-2.5,0.5) node {$\mQ_1$}; \draw (-1.5,0.5) node {$\cdots$};  \draw (-0.5,0.5) node{$\mQ_1$};
 \draw (-2.5, 1.6) node {$\vdots$};  \draw (-1.5, 1.5) node{$\cdots $};  \draw (-0.5, 1.6) node{$\vdots$};
  \draw (-2.5,2.5) node{$\mQ_1$}; \draw (-1.5,2.5) node{$\cdots$};  \draw (-0.5,2.5) node{$\mQ_1$};

 \draw (0,0) grid (3,3);
 \draw (0.5,0.5) node {$\mQ_1$}; \draw (1.5,0.5) node{$\cdots$};  \draw (2.5,0.5) node{$\mQ_1$};
 \draw (0.5, 1.6) node {$\vdots$};  \draw (1.5, 1.5) node{$\cdots $};  \draw (2.5, 1.6) node{$\vdots$};
  \draw (0.5,2.5) node {$\mQ_1$}; \draw (1.5,2.5) node {$\cdots$};  \draw (2.5,2.5) node{$\mQ_1$};
  
  \draw (3,0) grid (6,3);
 \draw (3.5,0.5) node{$\mQ_1$}; \draw (4.5,0.5) node{$\cdots$};  \draw (5.5,0.5) node{$\mQ_1$};
 \draw (3.5, 1.6) node{$\vdots$};  \draw (4.5, 1.5) node{$\cdots $};  \draw (5.5, 1.6) node{$\vdots$};
  \draw (3.5,2.5) node{$\mQ_1$}; \draw (4.5,2.5) node{$\cdots$};  \draw (5.5,2.5) node{$\mQ_1$};
  
   \draw (6,0) grid (9,3);
 \draw (6.5,0.5) node{$\mQ_1$}; \draw (7.5,0.5) node{$\cdots$};  \draw (8.5,0.5) node{$\mQ_1$};
 \draw (6.5, 1.6) node{$\vdots$};  \draw (7.5, 1.5) node{$\cdots $};  \draw (8.5, 1.6) node{$\vdots$};
  \draw (6.5,2.5) node{$\mQ_1$}; \draw (7.5,2.5) node{$\cdots$};  \draw (8.5,2.5) node{$\mQ_1$};

      \draw[dashed](10,1.5)--(14,1.5);
  
  \draw (15,0) grid (18,3);
 \draw (15.5,0.5) node{$\mQ_1$}; \draw (16.5,0.5) node{$\cdots$};  \draw (17.5,0.5) node{$\mQ_1$};
 \draw (15.5, 1.6) node{$\vdots$};  \draw (16.5, 1.5) node{$\cdots $};  \draw (17.5, 1.6) node{$\vdots$};
  \draw (15.5,2.5) node{$\mQ_1$}; \draw (16.5,2.5) node{$\cdots$};  \draw (17.5,2.5) node{$\mQ_1$};

    \draw (18,0) grid (21,3);
 \draw (18.5,0.5) node{$\mQ_1$}; \draw (19.5,0.5) node{$\cdots$};  \draw (20.5,0.5) node{$\mQ_1$};
 \draw (18.5, 1.6) node{$\vdots$};  \draw (19.5, 1.5) node{$\cdots $};  \draw (20.5, 1.6) node{$\vdots$};
  \draw (18.5,2.5) node{$\mQ_1$}; \draw (19.5,2.5) node{$\cdots$};  \draw (20.5,2.5) node{$\mQ_1$};


 \draw (-18,3) grid (-15,6);
 \draw (-17.5,3.5) node{$\mQ_1$}; \draw (-16.5,3.5) node{$\cdots$};  \draw (-15.5,3.5) node{$\mQ_1$};
 \draw (-17.5, 4.6) node{$\vdots$};  \draw (-16.5, 4.5) node{$\cdots $};  \draw (-15.5, 4.6) node{$\vdots$};
  \draw (-17.5,5.5) node{$\mQ_1$}; \draw (-16.5, 5.5) node{$\cdots$};  \draw (-15.5,5.5) node{$\mQ_1$};

    \draw (-15,3) grid (-12,6);
 \draw (-14.5,3.5) node{$\mQ_1$}; \draw (-13.5,3.5) node{$\cdots$};  \draw (-12.5,3.5) node{$\mQ_1$};
 \draw (-14.5, 4.6) node{$\vdots$};  \draw (-13.5, 4.5) node{$\cdots $};  \draw (-12.5, 4.6) node{$\vdots$};
  \draw (-14.5, 5.5) node{$\mQ_1$}; \draw (-13.5, 5.5) node{$\cdots$};  \draw (-12.5, 5.5) node{$\mQ_1$};

  \draw[dashed](-11,4.5)--(-7, 4.5);

\draw (-6,3) grid (-3,6);
 \draw (-5.5, 3.5) node{$\mQ_1$}; \draw (-4.5,3.5) node{$\cdots$};  \draw (-3.5, 3.5) node {$\mQ_1$};
 \draw (-5.5, 4.6) node{$\vdots$};  \draw (-4.5, 4.5) node {$\cdots $};  \draw (-3.5, 4.6) node {$\vdots$};
  \draw (-5.5, 5.5) node{$\mQ_1$}; \draw (-4.5, 5.5) node{$\cdots$};  \draw (-3.5, 5.5) node {$\mQ_1$};

\draw (-3,3) grid (0,6);
 \draw (-2.5, 3.5) node {$\mQ_1$}; \draw (-1.5, 3.5) node {$\cdots$};  \draw (-0.5,3.5) node{$\mQ_1$};
 \draw (-2.5, 4.6) node {$\vdots$};  \draw (-1.5, 4.5) node{$\cdots $};  \draw (-0.5, 4.6) node{$\vdots$};
  \draw (-2.5, 5.5) node{$\mQ_1$}; \draw (-1.5, 5.5) node{$\cdots$};  \draw (-0.5, 5.5) node{$\mQ_1$};

 \draw (0,3) grid (3,6);
 \draw (0.5, 3.5) node {$\mQ_1$}; \draw (1.5, 3.5) node{$\cdots$};  \draw (2.5, 3.5) node{$\mQ_1$};
 \draw (0.5, 4.6) node {$\vdots$};  \draw (1.5, 4.5) node{$\cdot $};  \draw (2.5, 4.6) node{$\vdots$};
  \draw (0.5, 5.5) node {$\mQ_1$}; \draw (1.5, 5.5) node {$\cdots$};  \draw (2.5, 5.5) node{$\mQ_1$};
  
  \draw (3,3) grid (6,6);
 \draw (3.5, 3.5) node{$\mQ_1$}; \draw (4.5,3.5) node{$\cdots$};  \draw (5.5,3.5) node{$\mQ_1$};
 \draw (3.5, 4.6) node{$\vdots$};  \draw (4.5, 4.5) node{$\cdots $};  \draw (5.5, 4.6) node{$\vdots$};
  \draw (3.5, 5.5) node{$\mQ_1$}; \draw (4.5,5.5) node{$\cdots$};  \draw (5.5, 5.5) node{$\mQ_1$};
  
   \draw (6,3) grid (9,6);
 \draw (6.5, 3.5) node{$\mQ_1$}; \draw (7.5,3.5) node{$\cdots$};  \draw (8.5, 3.5) node{$\mQ_1$};
 \draw (6.5, 4.6) node{$\vdots$};  \draw (7.5, 4.5) node{$\cdots $};  \draw (8.5, 4.6) node{$\vdots$};
  \draw (6.5, 5.5) node{$\mQ_1$}; \draw (7.5, 5.5) node{$\cdots$};  \draw (8.5, 5.5) node{$\mQ_1$};

      \draw[dashed](10, 4.5)--(14, 4.5);
  
  \draw (15,3) grid (18,6);
 \draw (15.5,3.5) node{$\mQ_1$}; \draw (16.5,3.5) node{$\cdots$};  \draw (17.5, 3.5) node{$\mQ_1$};
 \draw (15.5, 4.6) node{$\vdots$};  \draw (16.5, 4.5) node{$\cdots $};  \draw (17.5, 4.6) node{$\vdots$};
  \draw (15.5,5.5) node{$\mQ_1$}; \draw (16.5, 5.5) node{$\cdots$};  \draw (17.5, 5.5) node{$\mQ_1$};

    \draw (18,3) grid (21,6);
 \draw (18.5, 3.5) node{$\mQ_1$}; \draw (19.5, 3.5) node{$\cdots$};  \draw (20.5, 3.5) node{$\mQ_1$};
 \draw (18.5, 4.6) node{$\vdots$};  \draw (19.5, 4.5) node{$\cdots $};  \draw (20.5, 4.6) node{$\vdots$};
  \draw (18.5,5.5) node{$\mQ_1$}; \draw (19.5, 5.5) node{$\cdots$};  \draw (20.5, 5.5) node{$\mQ_1$};


 \draw (-18,6) grid (-15,9);
 \draw (-17.5,6.5) node{$\mQ_1$}; \draw (-16.5,6.5) node{$\cdots$};  \draw (-15.5,6.5) node{$\mQ_1$};
 \draw (-17.5, 7.6) node{$\vdots$};  \draw (-16.5, 7.5) node{$\cdots $};  \draw (-15.5, 7.6) node{$\vdots$};
  \draw (-17.5,8.5) node{$\mQ_1$}; \draw (-16.5, 8.5) node{$\cdots$};  \draw (-15.5,8.5) node{$\mQ_1$};

    \draw (-15,6) grid (-12,9);
 \draw (-14.5,6.5) node{$\mQ_1$}; \draw (-13.5,6.5) node{$\cdots$};  \draw (-12.5,6.5) node{$\mQ_1$};
 \draw (-14.5, 7.6) node{$\vdots$};  \draw (-13.5, 7.5) node{$\cdots $};  \draw (-12.5, 7.6) node{$\vdots$};
  \draw (-14.5, 8.5) node{$\mQ_1$}; \draw (-13.5, 8.5) node{$\cdots$};  \draw (-12.5, 8.5) node{$\mQ_1$};

  \draw[dashed](-11,7.5)--(-7, 7.5);

\draw (-6,6) grid (-3,9);
 \draw (-5.5, 6.5) node{$\mQ_1$}; \draw (-4.5,6.5) node{$\cdots$};  \draw (-3.5, 6.5) node {$\mQ_1$};
 \draw (-5.5, 7.6) node{$\vdots$};  \draw (-4.5, 7.5) node {$\cdots $};  \draw (-3.5, 7.6) node {$\vdots$};
  \draw (-5.5, 8.5) node{$\mQ_1$}; \draw (-4.5, 8.5) node{$\cdots$};  \draw (-3.5, 8.5) node {$\mQ_1$};

\draw (-3,6) grid (0,9);
 \draw (-2.5, 6.5) node {$\mQ_1$}; \draw (-1.5, 6.5) node {$\cdots$};  \draw (-0.5,6.5) node{$\mQ_1$};
 \draw (-2.5, 7.6) node {$\vdots$};  \draw (-1.5, 7.5) node{$\cdots $};  \draw (-0.5, 7.6) node{$\vdots$};
  \draw (-2.5, 8.5) node{$\mQ_1$}; \draw (-1.5, 8.5) node{$\cdots$};  \draw (-0.5, 8.5) node{$\mQ_1$};

 \draw (0,6) grid (3,9);
 \draw (0.5, 6.5) node {$\mQ_1$}; \draw (1.5, 6.5) node{$\cdots$};  \draw (2.5, 6.5) node{$\mQ_1$};
 \draw (0.5, 7.6) node {$\vdots$};  \draw (1.5, 7.5) node{$\cdot $};  \draw (2.5, 7.6) node{$\vdots$};
  \draw (0.5, 8.5) node {$\mQ_1$}; \draw (1.5, 8.5) node {$\cdots$};  \draw (2.5, 8.5) node{$\mQ_1$};
  
  \draw (3,6) grid (6,9);
 \draw (3.5, 6.5) node{$\mQ_1$}; \draw (4.5,6.5) node{$\cdots$};  \draw (5.5,6.5) node{$\mQ_1$};
 \draw (3.5, 7.6) node{$\vdots$};  \draw (4.5, 7.5) node{$\cdots $};  \draw (5.5, 7.6) node{$\vdots$};
  \draw (3.5, 8.5) node{$\mQ_1$}; \draw (4.5,8.5) node{$\cdots$};  \draw (5.5, 8.5) node{$\mQ_1$};
  
   \draw (6,6) grid (9,9);
 \draw (6.5, 6.5) node{$\mQ_1$}; \draw (7.5,6.5) node{$\cdots$};  \draw (8.5, 6.5) node{$\mQ_1$};
 \draw (6.5, 7.6) node{$\vdots$};  \draw (7.5, 7.5) node{$\cdots $};  \draw (8.5, 7.6) node{$\vdots$};
  \draw (6.5, 8.5) node{$\mQ_1$}; \draw (7.5, 8.5) node{$\cdots$};  \draw (8.5, 8.5) node{$\mQ_1$};

      \draw[dashed](10, 7.5)--(14, 7.5);
  
  \draw (15,6) grid (18,9);
 \draw (15.5,6.5) node{$\mQ_1$}; \draw (16.5,6.5) node{$\cdots$};  \draw (17.5, 6.5) node{$\mQ_1$};
 \draw (15.5, 7.6) node{$\vdots$};  \draw (16.5, 7.5) node{$\cdots $};  \draw (17.5, 7.6) node{$\vdots$};
  \draw (15.5,8.5) node{$\mQ_1$}; \draw (16.5, 8.5) node{$\cdots$};  \draw (17.5, 8.5) node{$\mQ_1$};

    \draw (18,6) grid (21,9);
 \draw (18.5, 6.5) node{$\mQ_1$}; \draw (19.5, 6.5) node{$\cdots$};  \draw (20.5, 6.5) node{$\mQ_1$};
 \draw (18.5, 7.6) node{$\vdots$};  \draw (19.5, 7.5) node{$\cdots $};  \draw (20.5, 7.6) node{$\vdots$};
  \draw (18.5,8.5) node{$\mQ_1$}; \draw (19.5, 8.5) node{$\cdots$};  \draw (20.5, 8.5) node{$\mQ_1$};


  \draw (-18,-3) grid (-15,0);
 \draw (-17.5,-2.5) node{$\mQ_1$}; \draw (-16.5,-2.5) node{$\cdots$};  \draw (-15.5,-2.5) node{$\mQ_1$};
 \draw (-17.5, -1.6) node{$\vdots$};  \draw (-16.5, -1.5) node{$\cdots $};  \draw (-15.5, -1.6) node{$\vdots$};
  \draw (-17.5,-0.5) node{$\mQ_1$}; \draw (-16.5,-0.5) node{$\cdots$};  \draw (-15.5,-0.5) node{$\mQ_1$};

    \draw (-15,-3) grid (-12,0);
 \draw (-14.5,-2.5) node{$\mQ_1$}; \draw (-13.5,-2.5) node{$\cdots$};  \draw (-12.5,-2.5) node{$\mQ_1$};
 \draw (-14.5, -1.6) node{$\vdots$};  \draw (-13.5, -1.5) node{$\cdots $};  \draw (-12.5, -1.6) node{$\vdots$};
  \draw (-14.5,-0.5) node{$\mQ_1$}; \draw (-13.5,-0.5) node{$\cdots$};  \draw (-12.5,-0.5) node{$\mQ_1$};
 
  \draw[dashed](-11,-1.5)--(-7,-1.5);

\draw (-6,-3) grid (-3,0);
 \draw (-5.5,-2.5) node{$\mQ_1$}; \draw (-4.5,-2.5) node{$\cdots$};  \draw (-3.5,-2.5) node {$\mQ_1$};
 \draw (-5.5, -1.6) node{$\vdots$};  \draw (-4.5, -1.5) node {$\cdots $};  \draw (-3.5, -1.6) node {$\vdots$};
  \draw (-5.5,-0.5) node{$\mQ_1$}; \draw (-4.5,-0.5) node{$\cdots$};  \draw (-3.5,-0.5) node {$\mQ_1$};

\draw (-3,-3) grid (0,0);
 \draw (-2.5,-2.5) node {$\mQ_1$}; \draw (-1.5,-2.5) node {$\cdots$};  \draw (-0.5,-2.5) node{$\mQ_1$};
 \draw (-2.5, -1.6) node {$\vdots$};  \draw (-1.5, -1.5) node{$\cdots $};  \draw (-0.5, -1.6) node{$\vdots$};
  \draw (-2.5,-0.5) node{$\mQ_1$}; \draw (-1.5,-0.5) node{$\cdots$};  \draw (-0.5,-0.5) node{$\mQ_1$};

 \draw (0,-3) grid (3,0);
 \draw (0.5,-2.5) node {$\mQ_1$}; \draw (1.5,-2.5) node{$\cdots$};  \draw (2.5,-2.5) node{$\mQ_1$};
 \draw (0.5, -1.6) node {$\vdots$};  \draw (1.5, -1.5) node{$\cdot $};  \draw (2.5, -1.6) node{$\vdots$};
  \draw (0.5,-0.5) node {$\mQ_1$}; \draw (1.5,-0.5) node {$\cdots$};  \draw (2.5,- 0.5) node{$\mQ_1$};
  
  \draw (3,-3) grid (6,0);
 \draw (3.5,-2.5) node{$\mQ_1$}; \draw (4.5,-2.5) node{$\cdots$};  \draw (5.5,-2.5) node{$\mQ_1$};
 \draw (3.5, -1.6) node{$\vdots$};  \draw (4.5, -1.5) node{$\cdots $};  \draw (5.5, -1.6) node{$\vdots$};
  \draw (3.5,-0.5) node{$\mQ_1$}; \draw (4.5,-0.5) node{$\cdots$};  \draw (5.5,-0.5) node{$\mQ_1$};
  
   \draw (6,-3) grid (9,0);
 \draw (6.5,-2.5) node{$\mQ_1$}; \draw (7.5,-2.5) node{$\cdots$};  \draw (8.5,-2.5) node{$\mQ_1$};
 \draw (6.5, -1.6) node{$\vdots$};  \draw (7.5, -1.5) node{$\cdots $};  \draw (8.5, -1.6) node{$\vdots$};
  \draw (6.5,-0.5) node{$\mQ_1$}; \draw (7.5,-0.5) node{$\cdots$};  \draw (8.5,-0.5) node{$\mQ_1$};

     \draw[dashed](10,-1.5)--(14,-1.5);
  
  \draw (15,-3) grid (18,0);
 \draw (15.5,-2.5) node{$\mQ_1$}; \draw (16.5,-2.5) node{$\cdots$};  \draw (17.5,-2.5) node{$\mQ_1$};
 \draw (15.5, -1.6) node{$\vdots$};  \draw (16.5, -1.5) node{$\cdots $};  \draw (17.5,-1.6) node{$\vdots$};
  \draw (15.5,-0.5) node{$\mQ_1$}; \draw (16.5,-0.5) node{$\cdots$};  \draw (17.5,-0.5) node{$\mQ_1$};

    \draw (18,-3) grid (21,0);
 \draw (18.5,-2.5) node{$\mQ_1$}; \draw (19.5,-2.5) node{$\cdots$};  \draw (20.5,-2.5) node{$\mQ_1$};
 \draw (18.5, -1.6) node{$\vdots$};  \draw (19.5, -1.5) node{$\cdots $};  \draw (20.5, -1.6) node{$\vdots$};
  \draw (18.5,-0.5) node{$\mQ_1$}; \draw (19.5,-0.5) node{$\cdots$};  \draw (20.5,-0.5) node{$\mQ_1$};


  \draw (-18,-6) grid (-15,-3);
 \draw (-17.5,-5.5) node{$\mQ_1$}; \draw (-16.5,-5.5) node{$\cdots$};  \draw (-15.5,-5.5) node{$\mQ_1$};
 \draw (-17.5, -4.6) node{$\vdots$};  \draw (-16.5, -4.5) node{$\cdots $};  \draw (-15.5, -4.6) node{$\vdots$};
  \draw (-17.5,-3.5) node{$\mQ_1$}; \draw (-16.5,-3.5) node{$\cdots$};  \draw (-15.5,-3.5) node{$\mQ_1$};

    \draw (-15,-6) grid (-12,-3);
 \draw (-14.5,-5.5) node{$\mQ_1$}; \draw (-13.5,-5.5) node{$\cdots$};  \draw (-12.5,-5.5) node{$\mQ_1$};
 \draw (-14.5, -4.6) node{$\vdots$};  \draw (-13.5, -4.5) node{$\cdots $};  \draw (-12.5, -4.6) node{$\vdots$};
  \draw (-14.5,-3.5) node{$\mQ_1$}; \draw (-13.5,-3.5) node{$\cdots$};  \draw (-12.5,-5.5) node{$\mQ_1$};

  \draw[dashed](-11,-4.5)--(-7,-4.5);

\draw (-6,-6) grid (-3,-3);
 \draw (-5.5,-5.5) node{$\mQ_1$}; \draw (-4.5,-5.5) node{$\cdots$};  \draw (-3.5,-5.5) node {$\mQ_1$};
 \draw (-5.5, -4.6) node{$\vdots$};  \draw (-4.5, -4.5) node {$\cdots $};  \draw (-3.5, -4.6) node {$\vdots$};
  \draw (-5.5,-3.5) node{$\mQ_1$}; \draw (-4.5,-3.5) node{$\cdots$};  \draw (-3.5,-3.5) node {$\mQ_1$};

\draw (-3,-6) grid (0,-3);
 \draw (-2.5,-5.5) node {$\mQ_1$}; \draw (-1.5,-5.5) node {$\cdots$};  \draw (-0.5,-5.5) node{$\mQ_1$};
 \draw (-2.5, -4.6) node {$\vdots$};  \draw (-1.5, -4.5) node{$\cdots $};  \draw (-0.5, -4.6) node{$\vdots$};
  \draw (-2.5,-3.5) node{$\mQ_1$}; \draw (-1.5,-3.5) node{$\cdots$};  \draw (-0.5,-3.5) node{$\mQ_1$};

 \draw (0,-6) grid (3,-3);
 \draw (0.5,-5.5) node {$\mQ_1$}; \draw (1.5,-5.5) node{$\cdots$};  \draw (2.5,-5.5) node{$\mQ_1$};
 \draw (0.5, -4.6) node {$\vdots$};  \draw (1.5, -4.5) node{$\cdot $};  \draw (2.5, -4.6) node{$\vdots$};
  \draw (0.5,-3.5) node {$\mQ_1$}; \draw (1.5,-3.5) node {$\cdots$};  \draw (2.5,- 3.5) node{$\mQ_1$};
  
  \draw (3,-6) grid (6,-3);
 \draw (3.5,-5.5) node{$\mQ_1$}; \draw (4.5,-5.5) node{$\cdots$};  \draw (5.5,-5.5) node{$\mQ_1$};
 \draw (3.5, -4.6) node{$\vdots$};  \draw (4.5, -4.5) node{$\cdots $};  \draw (5.5, -4.6) node{$\vdots$};
  \draw (3.5,-3.5) node{$\mQ_1$}; \draw (4.5,-3.5) node{$\cdots$};  \draw (5.5,-3.5) node{$\mQ_1$};
  
   \draw (6,-6) grid (9,-3);
 \draw (6.5,-5.5) node{$\mQ_1$}; \draw (7.5,-5.5) node{$\cdots$};  \draw (8.5,-5.5) node{$\mQ_1$};
 \draw (6.5, -4.6) node{$\vdots$};  \draw (7.5, -4.5) node{$\cdots $};  \draw (8.5, -4.6) node{$\vdots$};
  \draw (6.5,-3.5) node{$\mQ_1$}; \draw (7.5,-3.5) node{$\cdots$};  \draw (8.5,-3.5) node{$\mQ_1$};

     \draw[dashed](10,-4.5)--(14,-4.5);
  
  \draw (15,-6) grid (18,-3);
 \draw (15.5,-5.5) node{$\mQ_1$}; \draw (16.5,-5.5) node{$\cdots$};  \draw (17.5,-5.5) node{$\mQ_1$};
 \draw (15.5, -4.6) node{$\vdots$};  \draw (16.5, -4.5) node{$\cdots $};  \draw (17.5,-4.6) node{$\vdots$};
  \draw (15.5,-3.5) node{$\mQ_1$}; \draw (16.5,-3.5) node{$\cdots$};  \draw (17.5,-3.5) node{$\mQ_1$};

    \draw (18,-6) grid (21,-3);
 \draw (18.5,-5.5) node{$\mQ_1$}; \draw (19.5,-5.5) node{$\cdots$};  \draw (20.5,-5.5) node{$\mQ_1$};
 \draw (18.5, -4.6) node{$\vdots$};  \draw (19.5, -4.5) node{$\cdots $};  \draw (20.5, -4.6) node{$\vdots$};
  \draw (18.5,-3.5) node{$\mQ_1$}; \draw (19.5,-3.5) node{$\cdots$};  \draw (20.5,-3.5) node{$\mQ_1$};


  \draw (-18,18) grid (-15,21);
 \draw (-17.5,18.5) node{$\mQ_1$}; \draw (-16.5,18.5) node{$\cdots$};  \draw (-15.5,18.5) node{$\mQ_1$};
 \draw (-17.5, 19.6) node{$\vdots$};  \draw (-16.5, 19.5) node{$\cdots $};  \draw (-15.5, 19.6) node{$\vdots$};
  \draw (-17.5, 20.5) node{$\mQ_1$}; \draw (-16.5, 20.5) node{$\cdots$};  \draw (-15.5, 20.5) node{$\mQ_1$};

    \draw (-15,18) grid (-12,21);
 \draw (-14.5,18.5) node{$\mQ_1$}; \draw (-13.5,18.5) node{$\cdots$};  \draw (-12.5,18.5) node{$\mQ_1$};
 \draw (-14.5, 19.6) node{$\vdots$};  \draw (-13.5, 19.5) node{$\cdots $};  \draw (-12.5, 19.6) node{$\vdots$};
  \draw (-14.5, 20.5) node{$\mQ_1$}; \draw (-13.5,20.5) node{$\cdots$};  \draw (-12.5, 20.5) node{$\mQ_1$};

  \draw[dashed](-11,19.5)--(-7, 19.5);

\draw (-6,18) grid (-3,21);
 \draw (-5.5,18.5) node{$\mQ_1$}; \draw (-4.5,18.5) node{$\cdots$};  \draw (-3.5,18.5) node {$\mQ_1$};
 \draw (-5.5, 19.6) node{$\vdots$};  \draw (-4.5, 19.5) node {$\cdots $};  \draw (-3.5, 19.6) node {$\vdots$};
  \draw (-5.5,20.5) node{$\mQ_1$}; \draw (-4.5,20.5) node{$\cdots$};  \draw (-3.5, 20.5) node {$\mQ_1$};

\draw (-3,18) grid (0,21);
 \draw (-2.5,18.5) node {$\mQ_1$}; \draw (-1.5,18.5) node {$\cdots$};  \draw (-0.5,18.5) node{$\mQ_1$};
 \draw (-2.5, 19.6) node {$\vdots$};  \draw (-1.5, 19.5) node{$\cdots $};  \draw (-0.5, 19.6) node{$\vdots$};
  \draw (-2.5,20.5) node{$\mQ_1$}; \draw (-1.5, 20.5) node{$\cdots$};  \draw (-0.5, 20.5) node{$\mQ_1$};

 \draw (0,18) grid (3,21);
 \draw (0.5,18.5) node {$\mQ_1$}; \draw (1.5,18.5) node{$\cdots$};  \draw (2.5,18.5) node{$\mQ_1$};
 \draw (0.5, 19.6) node {$\vdots$};  \draw (1.5, 19.5) node{$\cdot $};  \draw (2.5, 19.6) node{$\vdots$};
  \draw (0.5, 20.5) node {$\mQ_1$}; \draw (1.5, 20.5) node {$\cdots$};  \draw (2.5, 20.5) node{$\mQ_1$};
  
  \draw (3,18) grid (6,21);
 \draw (3.5,18.5) node{$\mQ_1$}; \draw (4.5,18.5) node{$\cdots$};  \draw (5.5,18.5) node{$\mQ_1$};
 \draw (3.5, 19.6) node{$\vdots$};  \draw (4.5, 19.5) node{$\cdots $};  \draw (5.5, 19.6) node{$\vdots$};
  \draw (3.5,20.5) node{$\mQ_1$}; \draw (4.5, 20.5) node{$\cdots$};  \draw (5.5, 20.5) node{$\mQ_1$};
  
   \draw (6,18) grid (9,21);
 \draw (6.5,18.5) node{$\mQ_1$}; \draw (7.5,18.5) node{$\cdots$};  \draw (8.5, 18.5) node{$\mQ_1$};
 \draw (6.5, 19.6) node{$\vdots$};  \draw (7.5, 19.5) node{$\cdots $};  \draw (8.5, 19.6) node{$\vdots$};
  \draw (6.5, 20.5) node{$\mQ_1$}; \draw (7.5, 20.5) node{$\cdots$};  \draw (8.5, 20.5) node{$\mQ_1$};

         \draw[dashed](10, 19.5)--(14, 19.5);
  
  \draw (15,18) grid (18,21);
 \draw (15.5,18.5) node{$\mQ_2$}; \draw (16.5,18.5) node{$\cdots$};  \draw (17.5,18.5) node{$\mQ_2$};
 \draw (15.5, 19.6) node{$\vdots$};  \draw (16.5, 19.5) node{$\cdots $};  \draw (17.5, 19.6) node{$\vdots$};
  \draw (15.5, 20.5) node{$\mQ_2$}; \draw (16.5, 20.5) node{$\cdots$};  \draw (17.5, 20.5) node{$\mQ_2$};

    \draw (18,18) grid (21,21);
 \draw (18.5,18.5) node{$\mQ_1$}; \draw (19.5,18.5) node{$\cdots$};  \draw (20.5,18.5) node{$\mQ_1$};
 \draw (18.5, 19.6) node{$\vdots$};  \draw (19.5, 19.5) node{$\cdots $};  \draw (20.5, 19.6) node{$\vdots$};
  \draw (18.5, 20.5) node{$\mQ_1$}; \draw (19.5, 20.5) node{$\cdots$};  \draw (20.5, 20.5) node{$\mQ_1$};
 \draw (-18,15) grid (-15,18);
 \draw (-17.5,15.5) node{$\mQ_1$}; \draw (-16.5,15.5) node{$\cdots$};  \draw (-15.5,15.5) node{$\mQ_1$};
 \draw (-17.5, 16.6) node{$\vdots$};  \draw (-16.5, 16.5) node{$\cdots $};  \draw (-15.5, 16.6) node{$\vdots$};
  \draw (-17.5,17.5) node{$\mQ_1$}; \draw (-16.5, 17.5) node{$\cdots$};  \draw (-15.5, 17.5) node{$\mQ_1$};

    \draw (-15,15) grid (-12,18);
 \draw (-14.5,15.5) node{$\mQ_1$}; \draw (-13.5,15.5) node{$\cdots$};  \draw (-12.5,15.5) node{$\mQ_1$};
 \draw (-14.5, 16.6) node{$\vdots$};  \draw (-13.5, 16.5) node{$\cdots $};  \draw (-12.5, 16.6) node{$\vdots$};
  \draw (-14.5, 17.5) node{$\mQ_1$}; \draw (-13.5,17.5) node{$\cdots$};  \draw (-12.5,17.5) node{$\mQ_1$};

  \draw[dashed](-11,16.5)--(-7,16.5);

\draw (-6,15) grid (-3,18);
 \draw (-5.5,15.5) node{$\mQ_1$}; \draw (-4.5,15.5) node{$\cdots$};  \draw (-3.5,15.5) node {$\mQ_1$};
 \draw (-5.5, 16.6) node{$\vdots$};  \draw (-4.5, 16.5) node {$\cdots $};  \draw (-3.5, 16.6) node {$\vdots$};
  \draw (-5.5,17.5) node{$\mQ_1$}; \draw (-4.5,17.5) node{$\cdots$};  \draw (-3.5, 17.5) node {$\mQ_1$};

\draw (-3,15) grid (0,18);
 \draw (-2.5,15.5) node {$\mQ_1$}; \draw (-1.5,15.5) node {$\cdots$};  \draw (-0.5,15.5) node{$\mQ_1$};
 \draw (-2.5, 16.6) node {$\vdots$};  \draw (-1.5, 16.5) node{$\cdots $};  \draw (-0.5, 16.6) node{$\vdots$};
  \draw (-2.5,17.5) node{$\mQ_1$}; \draw (-1.5, 17.5) node{$\cdots$};  \draw (-0.5,17.5) node{$\mQ_1$};

 \draw (0,15) grid (3,18);
 \draw (0.5,15.5) node {$\mQ_1$}; \draw (1.5,15.5) node{$\cdots$};  \draw (2.5,15.5) node{$\mQ_1$};
 \draw (0.5, 16.6) node {$\vdots$};  \draw (1.5, 16.5) node{$\cdot $};  \draw (2.5, 16.6) node{$\vdots$};
  \draw (0.5,17.5) node {$\mQ_1$}; \draw (1.5, 17.5) node {$\cdots$};  \draw (2.5,17.5) node{$\mQ_1$};
  
  \draw (3,15) grid (6,18);
 \draw (3.5,15.5) node{$\mQ_1$}; \draw (4.5,15.5) node{$\cdots$};  \draw (5.5,15.5) node{$\mQ_1$};
 \draw (3.5, 16.6) node{$\vdots$};  \draw (4.5, 16.5) node{$\cdots $};  \draw (5.5, 16.6) node{$\vdots$};
  \draw (3.5,17.5) node{$\mQ_1$}; \draw (4.5,17.5) node{$\cdots$};  \draw (5.5,17.5) node{$\mQ_1$};
  
   \draw (6,15) grid (9,18);
 \draw (6.5,15.5) node{$\mQ_1$}; \draw (7.5,15.5) node{$\cdots$};  \draw (8.5, 15.5) node{$\mQ_1$};
 \draw (6.5, 16.6) node{$\vdots$};  \draw (7.5, 16.5) node{$\cdots $};  \draw (8.5, 16.6) node{$\vdots$};
  \draw (6.5,17.5) node{$\mQ_1$}; \draw (7.5,17.5) node{$\cdots$};  \draw (8.5,17.5) node{$\mQ_1$};

          \draw[dashed](10, 16.5)--(14, 16.5);
  
  \draw (15,15) grid (18,18);
 \draw (15.5,15.5) node{$\mQ_1$}; \draw (16.5,15.5) node{$\cdots$};  \draw (17.5,15.5) node{$\mQ_1$};
 \draw (15.5, 16.6) node{$\vdots$};  \draw (16.5, 16.5) node{$\cdots $};  \draw (17.5, 16.6) node{$\vdots$};
  \draw (15.5, 17.5) node{$\mQ_1$}; \draw (16.5,17.5) node{$\cdots$};  \draw (17.5,17.5) node{$\mQ_1$};

    \draw (18,15) grid (21,18);
 \draw (18.5,15.5) node{$\mQ_1$}; \draw (19.5,15.5) node{$\cdots$};  \draw (20.5,15.5) node{$\mQ_1$};
 \draw (18.5, 16.6) node{$\vdots$};  \draw (19.5, 16.5) node{$\cdots $};  \draw (20.5, 16.6) node{$\vdots$};
  \draw (18.5,17.5) node{$\mQ_1$}; \draw (19.5,17.5) node{$\cdots$};  \draw (20.5,17.5) node{$\mQ_1$};


\draw (-18,12) grid (-15,15);
 \draw (-17.5,12.5) node{$\mQ_1$}; \draw (-16.5,12.5) node{$\cdots$};  \draw (-15.5,12.5) node{$\mQ_1$};
 \draw (-17.5, 13.6) node{$\vdots$};  \draw (-16.5, 13.5) node{$\cdots $};  \draw (-15.5, 13.6) node{$\vdots$};
  \draw (-17.5,14.5) node{$\mQ_1$}; \draw (-16.5, 14.5) node{$\cdots$};  \draw (-15.5, 14.5) node{$\mQ_1$};

    \draw (-15,12) grid (-12,15);
 \draw (-14.5,12.5) node{$\mQ_1$}; \draw (-13.5,12.5) node{$\cdots$};  \draw (-12.5,12.5) node{$\mQ_1$};
 \draw (-14.5, 13.6) node{$\vdots$};  \draw (-13.5, 13.5) node{$\cdots $};  \draw (-12.5, 13.6) node{$\vdots$};
  \draw (-14.5, 14.5) node{$\mQ_1$}; \draw (-13.5,14.5) node{$\cdots$};  \draw (-12.5,14.5) node{$\mQ_1$};

  \draw[dashed](-11,13.5)--(-7,13.5);

\draw (-6,12) grid (-3,15);
 \draw (-5.5,12.5) node{$\mQ_1$}; \draw (-4.5,12.5) node{$\cdots$};  \draw (-3.5,12.5) node {$\mQ_1$};
 \draw (-5.5, 13.6) node{$\vdots$};  \draw (-4.5, 13.5) node {$\cdots $};  \draw (-3.5, 13.6) node {$\vdots$};
  \draw (-5.5,14.5) node{$\mQ_1$}; \draw (-4.5,14.5) node{$\cdots$};  \draw (-3.5, 14.5) node {$\mQ_1$};

\draw (-3,12) grid (0,15);
 \draw (-2.5,12.5) node {$\mQ_1$}; \draw (-1.5,12.5) node {$\cdots$};  \draw (-0.5,12.5) node{$\mQ_1$};
 \draw (-2.5, 13.6) node {$\vdots$};  \draw (-1.5, 13.5) node{$\cdots $};  \draw (-0.5, 13.6) node{$\vdots$};
  \draw (-2.5,14.5) node{$\mQ_1$}; \draw (-1.5, 14.5) node{$\cdots$};  \draw (-0.5,14.5) node{$\mQ_1$};

 \draw (0,12) grid (3,15);
 \draw (0.5,12.5) node {$\mQ_1$}; \draw (1.5,12.5) node{$\cdots$};  \draw (2.5,12.5) node{$\mQ_1$};
 \draw (0.5, 13.6) node {$\vdots$};  \draw (1.5, 13.5) node{$\cdot $};  \draw (2.5, 13.6) node{$\vdots$};
  \draw (0.5,14.5) node {$\mQ_1$}; \draw (1.5, 14.5) node {$\cdots$};  \draw (2.5,14.5) node{$\mQ_1$};
  
  \draw (3,12) grid (6,15);
 \draw (3.5,12.5) node{$\mQ_1$}; \draw (4.5,12.5) node{$\cdots$};  \draw (5.5,12.5) node{$\mQ_1$};
 \draw (3.5, 13.6) node{$\vdots$};  \draw (4.5, 13.5) node{$\cdots $};  \draw (5.5, 13.6) node{$\vdots$};
  \draw (3.5,14.5) node{$\mQ_1$}; \draw (4.5,14.5) node{$\cdots$};  \draw (5.5,14.5) node{$\mQ_1$};
  
   \draw (6,12) grid (9,15);
 \draw (6.5,12.5) node{$\mQ_1$}; \draw (7.5,12.5) node{$\cdots$};  \draw (8.5, 12.5) node{$\mQ_1$};
 \draw (6.5, 13.6) node{$\vdots$};  \draw (7.5, 13.5) node{$\cdots $};  \draw (8.5, 13.6) node{$\vdots$};
  \draw (6.5,14.5) node{$\mQ_1$}; \draw (7.5,14.5) node{$\cdots$};  \draw (8.5,14.5) node{$\mQ_1$};

          \draw[dashed](10, 13.5)--(14, 13.5);
  
  \draw (15,12) grid (18,15);
 \draw (15.5,12.5) node{$\mQ_1$}; \draw (16.5,12.5) node{$\cdots$};  \draw (17.5,12.5) node{$\mQ_1$};
 \draw (15.5, 13.6) node{$\vdots$};  \draw (16.5, 13.5) node{$\cdots $};  \draw (17.5, 13.6) node{$\vdots$};
  \draw (15.5, 14.5) node{$\mQ_1$}; \draw (16.5,14.5) node{$\cdots$};  \draw (17.5,14.5) node{$\mQ_1$};

    \draw (18,12) grid (21,15);
 \draw (18.5,12.5) node{$\mQ_1$}; \draw (19.5,12.5) node{$\cdots$};  \draw (20.5,12.5) node{$\mQ_1$};
 \draw (18.5, 13.6) node{$\vdots$};  \draw (19.5, 13.5) node{$\cdots $};  \draw (20.5, 13.6) node{$\vdots$};
  \draw (18.5,14.5) node{$\mQ_1$}; \draw (19.5,14.5) node{$\cdots$};  \draw (20.5,14.5) node{$\mQ_1$};

 \draw (-18,-18) grid (-15,-15);
 \draw (-17.5,-17.5) node{$\mQ_1$}; \draw (-16.5,-17.5) node{$\cdots$};  \draw (-15.5,-17.5) node{$\mQ_1$};
 \draw (-17.5, -16.6) node{$\vdots$};  \draw (-16.5, -16.5) node{$\cdots $};  \draw (-15.5, -16.6) node{$\vdots$};
  \draw (-17.5,-15.5) node{$\mQ_1$}; \draw (-16.5, -15.5) node{$\cdots$};  \draw (-15.5, -15.5) node{$\mQ_1$};

    \draw (-15,-18) grid (-12,-15);
 \draw (-14.5,-17.5) node{$\mQ_2$}; \draw (-13.5,-17.5) node{$\cdots$};  \draw (-12.5,-17.5) node{$\mQ_2$};
 \draw (-14.5, -16.6) node{$\vdots$};  \draw (-13.5, -16.5) node{$\cdots $};  \draw (-12.5, -16.6) node{$\vdots$};
  \draw (-14.5, -15.5) node{$\mQ_2$}; \draw (-13.5,-15.5) node{$\cdots$};  \draw (-12.5,-15.5) node{$\mQ_2$};

  \draw[dashed](-11,-16.5)--(-7,-16.5);

\draw (-6,-18) grid (-3,-15);
 \draw (-5.5,-17.5) node{$\mQ_1$}; \draw (-4.5,-17.5) node{$\cdots$};  \draw (-3.5,-17.5) node {$\mQ_1$};
 \draw (-5.5, -16.6) node{$\vdots$};  \draw (-4.5, -16.5) node {$\cdots $};  \draw (-3.5, -16.6) node {$\vdots$};
  \draw (-5.5,-15.5) node{$\mQ_1$}; \draw (-4.5,-15.5) node{$\cdots$};  \draw (-3.5, -15.5) node {$\mQ_1$};

\draw (-3,-18) grid (0,-15);
 \draw (-2.5,-17.5) node {$\mQ_2$}; \draw (-1.5,-17.5) node {$\cdots$};  \draw (-0.5,-17.5) node{$\mQ_2$};
 \draw (-2.5, -16.6) node {$\vdots$};  \draw (-1.5, -16.5) node{$\cdots $};  \draw (-0.5, -16.6) node{$\vdots$};
  \draw (-2.5,-15.5) node{$\mQ_2$}; \draw (-1.5, -15.5) node{$\cdots$};  \draw (-0.5,-15.5) node{$\mQ_2$};

 \draw (0,-18) grid (3,-15);
 \draw (0.5,-17.5) node {$\mQ_1$}; \draw (1.5,-17.5) node{$\cdots$};  \draw (2.5,-17.5) node{$\mQ_1$};
 \draw (0.5, -16.6) node {$\vdots$};  \draw (1.5, -16.5) node{$\cdot $};  \draw (2.5, -16.6) node{$\vdots$};
  \draw (0.5,-15.5) node {$\mQ_1$}; \draw (1.5, -15.5) node {$\cdots$};  \draw (2.5,-15.5) node{$\mQ_1$};
  
  \draw (3,-18) grid (6,-15);
 \draw (3.5,-17.5) node{$\mQ_2$}; \draw (4.5,-17.5) node{$\cdots$};  \draw (5.5,-17.5) node{$\mQ_2$};
 \draw (3.5, -16.6) node{$\vdots$};  \draw (4.5, -16.5) node{$\cdots $};  \draw (5.5, -16.6) node{$\vdots$};
  \draw (3.5,-15.5) node{$\mQ_2$}; \draw (4.5,-15.5) node{$\cdots$};  \draw (5.5,-15.5) node{$\mQ_2$};
  
   \draw (6,-18) grid (9,-15);
 \draw (6.5,-17.5) node{$\mQ_1$}; \draw (7.5,-17.5) node{$\cdots$};  \draw (8.5, -17.5) node{$\mQ_1$};
 \draw (6.5, -16.6) node{$\vdots$};  \draw (7.5, -16.5) node{$\cdots $};  \draw (8.5, -16.6) node{$\vdots$};
  \draw (6.5,-15.5) node{$\mQ_1$}; \draw (7.5,-15.5) node{$\cdots$};  \draw (8.5,-15.5) node{$\mQ_1$};
   
       \draw[dashed](10,-16.5)--(14,-16.5);
  
  \draw (15,-18) grid (18,-15);
 \draw (15.5,-17.5) node{$\mQ_2$}; \draw (16.5,-17.5) node{$\cdots$};  \draw (17.5,-17.5) node{$\mQ_2$};
 \draw (15.5, -16.6) node{$\vdots$};  \draw (16.5, -16.5) node{$\cdots $};  \draw (17.5, -16.6) node{$\vdots$};
  \draw (15.5, -15.5) node{$\mQ_2$}; \draw (16.5,-15.5) node{$\cdots$};  \draw (17.5,-15.5) node{$\mQ_2$};

    \draw (18,-18) grid (21,-15);
 \draw (18.5,-17.5) node{$\mQ_1$}; \draw (19.5,-17.5) node{$\cdots$};  \draw (20.5,-17.5) node{$\mQ_1$};
 \draw (18.5, -16.6) node{$\vdots$};  \draw (19.5, -16.5) node{$\cdots $};  \draw (20.5, -16.6) node{$\vdots$};
  \draw (18.5,-15.5) node{$\mQ_1$}; \draw (19.5,-15.5) node{$\cdots$};  \draw (20.5,-15.5) node{$\mQ_1$};
 \draw (-18,-12) grid (-15,-9);
 \draw (-17.5,-11.5) node{$\mQ_1$}; \draw (-16.5,-11.5) node{$\cdots$};  \draw (-15.5,-11.5) node{$\mQ_1$};
 \draw (-17.5, -10.6) node{$\vdots$};  \draw (-16.5, -10.5) node{$\cdots $};  \draw (-15.5, -10.6) node{$\vdots$};
  \draw (-17.5,-9.5) node{$\mQ_1$}; \draw (-16.5, -9.5) node{$\cdots$};  \draw (-15.5, -9.5) node{$\mQ_1$};

    \draw (-15,-12) grid (-12,-9);
 \draw (-14.5,-11.5) node{$\mQ_1$}; \draw (-13.5,-11.5) node{$\cdots$};  \draw (-12.5,-11.5) node{$\mQ_1$};
 \draw (-14.5, -10.6) node{$\vdots$};  \draw (-13.5, -10.5) node{$\cdots $};  \draw (-12.5, -10.6) node{$\vdots$};
  \draw (-14.5, -9.5) node{$\mQ_1$}; \draw (-13.5,-9.5) node{$\cdots$};  \draw (-12.5,-9.5) node{$\mQ_1$};

 \draw[dashed](-11,-10.5)--(-7,-10.5);

\draw (-6,-12) grid (-3,-9);
 \draw (-5.5,-11.5) node{$\mQ_1$}; \draw (-4.5,-11.5) node{$\cdots$};  \draw (-3.5,-11.5) node {$\mQ_1$};
 \draw (-5.5, -10.6) node{$\vdots$};  \draw (-4.5, -10.5) node {$\cdots $};  \draw (-3.5, -10.6) node {$\vdots$};
  \draw (-5.5,-9.5) node{$\mQ_1$}; \draw (-4.5,-9.5) node{$\cdots$};  \draw (-3.5, -9.5) node {$\mQ_1$};

\draw (-3,-12) grid (0,-9);
 \draw (-2.5,-11.5) node {$\mQ_1$}; \draw (-1.5,-11.5) node {$\cdots$};  \draw (-0.5,-11.5) node{$\mQ_1$};
 \draw (-2.5, -10.6) node {$\vdots$};  \draw (-1.5, -10.5) node{$\cdots $};  \draw (-0.5, -10.6) node{$\vdots$};
  \draw (-2.5,-9.5) node{$\mQ_1$}; \draw (-1.5, -9.5) node{$\cdots$};  \draw (-0.5,-9.5) node{$\mQ_1$};

 \draw (0,-12) grid (3,-9);
 \draw (0.5,-11.5) node {$\mQ_1$}; \draw (1.5,-11.5) node{$\cdots$};  \draw (2.5,-11.5) node{$\mQ_1$};
 \draw (0.5, -10.6) node {$\vdots$};  \draw (1.5, -10.5) node{$\cdot $};  \draw (2.5, -10.6) node{$\vdots$};
  \draw (0.5,-9.5) node {$\mQ_1$}; \draw (1.5, -9.5) node {$\cdots$};  \draw (2.5,-9.5) node{$\mQ_1$};
  
  \draw (3,-12) grid (6,-9);
 \draw (3.5,-11.5) node{$\mQ_1$}; \draw (4.5,-11.5) node{$\cdots$};  \draw (5.5,-11.5) node{$\mQ_1$};
 \draw (3.5, -10.6) node{$\vdots$};  \draw (4.5, -10.5) node{$\cdots $};  \draw (5.5, -10.6) node{$\vdots$};
  \draw (3.5,-9.5) node{$\mQ_1$}; \draw (4.5,-9.5) node{$\cdots$};  \draw (5.5,-9.5) node{$\mQ_1$};
  
   \draw (6,-12) grid (9,-9);
 \draw (6.5,-11.5) node{$\mQ_1$}; \draw (7.5,-11.5) node{$\cdots$};  \draw (8.5, -11.5) node{$\mQ_1$};
 \draw (6.5, -10.6) node{$\vdots$};  \draw (7.5, -10.5) node{$\cdots $};  \draw (8.5, -10.6) node{$\vdots$};
  \draw (6.5,-9.5) node{$\mQ_1$}; \draw (7.5,-9.5) node{$\cdots$};  \draw (8.5,-9.5) node{$\mQ_1$};

 \draw[dashed](10,-10.5)--(14,-10.5);
 
  \draw (15,-12) grid (18,-9);
 \draw (15.5,-11.5) node{$\mQ_1$}; \draw (16.5,-11.5) node{$\cdots$};  \draw (17.5,-11.5) node{$\mQ_1$};
 \draw (15.5, -10.6) node{$\vdots$};  \draw (16.5, -10.5) node{$\cdots $};  \draw (17.5, -10.6) node{$\vdots$};
  \draw (15.5, -9.5) node{$\mQ_1$}; \draw (16.5,-9.5) node{$\cdots$};  \draw (17.5,-9.5) node{$\mQ_1$};

    \draw (18,-12) grid (21,-9);
 \draw (18.5,-11.5) node{$\mQ_1$}; \draw (19.5,-11.5) node{$\cdots$};  \draw (20.5,-11.5) node{$\mQ_1$};
 \draw (18.5, -10.6) node{$\vdots$};  \draw (19.5, -10.5) node{$\cdots $};  \draw (20.5, -10.6) node{$\vdots$};
  \draw (18.5,-9.5) node{$\mQ_1$}; \draw (19.5,-9.5) node{$\cdots$};  \draw (20.5,-9.5) node{$\mQ_1$};

 \draw (-18,-15) grid (-15,-12);
 \draw (-17.5,-14.5) node{$\mQ_1$}; \draw (-16.5,-14.5) node{$\cdots$};  \draw (-15.5,-14.5) node{$\mQ_1$};
 \draw (-17.5, -13.6) node{$\vdots$};  \draw (-16.5, -13.5) node{$\cdots $};  \draw (-15.5, -13.6) node{$\vdots$};
  \draw (-17.5,-12.5) node{$\mQ_1$}; \draw (-16.5, -12.5) node{$\cdots$};  \draw (-15.5, -12.5) node{$\mQ_1$};

    \draw (-15,-15) grid (-12,-12);
 \draw (-14.5,-14.5) node{$\mQ_1$}; \draw (-13.5,-14.5) node{$\cdots$};  \draw (-12.5,-14.5) node{$\mQ_1$};
 \draw (-14.5, -13.6) node{$\vdots$};  \draw (-13.5, -13.5) node{$\cdots $};  \draw (-12.5, -13.6) node{$\vdots$};
  \draw (-14.5, -12.5) node{$\mQ_1$}; \draw (-13.5,-12.5) node{$\cdots$};  \draw (-12.5,-12.5) node{$\mQ_1$};

 \draw[dashed](-11,-13.5)--(-7,-13.5);

\draw (-6,-15) grid (-3,-12);
 \draw (-5.5,-14.5) node{$\mQ_1$}; \draw (-4.5,-14.5) node{$\cdots$};  \draw (-3.5,-14.5) node {$\mQ_1$};
 \draw (-5.5, -13.6) node{$\vdots$};  \draw (-4.5, -13.5) node {$\cdots $};  \draw (-3.5, -13.6) node {$\vdots$};
  \draw (-5.5,-12.5) node{$\mQ_1$}; \draw (-4.5,-12.5) node{$\cdots$};  \draw (-3.5, -12.5) node {$\mQ_1$};

\draw (-3,-15) grid (0,-12);
 \draw (-2.5,-14.5) node {$\mQ_1$}; \draw (-1.5,-14.5) node {$\cdots$};  \draw (-0.5,-14.5) node{$\mQ_1$};
 \draw (-2.5, -13.6) node {$\vdots$};  \draw (-1.5, -13.5) node{$\cdots $};  \draw (-0.5, -13.6) node{$\vdots$};
  \draw (-2.5,-12.5) node{$\mQ_1$}; \draw (-1.5, -12.5) node{$\cdots$};  \draw (-0.5,-12.5) node{$\mQ_1$};

 \draw (0,-15) grid (3,-12);
 \draw (0.5,-14.5) node {$\mQ_1$}; \draw (1.5,-14.5) node{$\cdots$};  \draw (2.5,-14.5) node{$\mQ_1$};
 \draw (0.5, -13.6) node {$\vdots$};  \draw (1.5, -13.5) node{$\cdot $};  \draw (2.5, -13.6) node{$\vdots$};
  \draw (0.5,-12.5) node {$\mQ_1$}; \draw (1.5, -12.5) node {$\cdots$};  \draw (2.5,-12.5) node{$\mQ_1$};
  
  \draw (3,-15) grid (6,-12);
 \draw (3.5,-14.5) node{$\mQ_1$}; \draw (4.5,-14.5) node{$\cdots$};  \draw (5.5,-14.5) node{$\mQ_1$};
 \draw (3.5, -13.6) node{$\vdots$};  \draw (4.5, -13.5) node{$\cdots $};  \draw (5.5, -13.6) node{$\vdots$};
  \draw (3.5,-12.5) node{$\mQ_1$}; \draw (4.5,-12.5) node{$\cdots$};  \draw (5.5,-12.5) node{$\mQ_1$};
  
   \draw (6,-15) grid (9,-12);
 \draw (6.5,-14.5) node{$\mQ_1$}; \draw (7.5,-14.5) node{$\cdots$};  \draw (8.5, -14.5) node{$\mQ_1$};
 \draw (6.5, -13.6) node{$\vdots$};  \draw (7.5, -13.5) node{$\cdots $};  \draw (8.5, -13.6) node{$\vdots$};
  \draw (6.5,-12.5) node{$\mQ_1$}; \draw (7.5,-12.5) node{$\cdots$};  \draw (8.5,-12.5) node{$\mQ_1$};

 \draw[dashed](10,-13.5)--(14,-13.5);
 
  \draw (15,-15) grid (18,-12);
 \draw (15.5,-14.5) node{$\mQ_1$}; \draw (16.5,-14.5) node{$\cdots$};  \draw (17.5,-14.5) node{$\mQ_1$};
 \draw (15.5, -13.6) node{$\vdots$};  \draw (16.5, -13.5) node{$\cdots $};  \draw (17.5, -13.6) node{$\vdots$};
  \draw (15.5, -12.5) node{$\mQ_1$}; \draw (16.5,-12.5) node{$\cdots$};  \draw (17.5,-12.5) node{$\mQ_1$};

    \draw (18,-15) grid (21,-12);
 \draw (18.5,-14.5) node{$\mQ_1$}; \draw (19.5,-14.5) node{$\cdots$};  \draw (20.5,-14.5) node{$\mQ_1$};
 \draw (18.5, -13.6) node{$\vdots$};  \draw (19.5, -13.5) node{$\cdots $};  \draw (20.5, -13.6) node{$\vdots$};
  \draw (18.5,-12.5) node{$\mQ_1$}; \draw (19.5,-12.5) node{$\cdots$};  \draw (20.5,-12.5) node{$\mQ_1$};

\draw (-18,9) grid (-15,12);
 \draw (-17.5,9.5) node{$\mQ_1$}; \draw (-16.5, 9.5) node{$\cdots$};  \draw (-15.5, 9.5) node{$\mQ_1$};
 \draw (-17.5, 10.6) node{$\vdots$};  \draw (-16.5, 10.5) node{$\cdots $};  \draw (-15.5, 10.6) node{$\vdots$};
  \draw (-17.5,11.5) node{$\mQ_1$}; \draw (-16.5, 11.5) node{$\cdots$};  \draw (-15.5, 11.5) node{$\mQ_1$};

    \draw (-15, 9) grid (-12,12);
 \draw (-14.5, 9.5) node{$\mQ_1$}; \draw (-13.5, 9.5) node{$\cdots$};  \draw (-12.5, 9.5) node{$\mQ_1$};
 \draw (-14.5, 10.6) node{$\vdots$};  \draw (-13.5, 10.5) node{$\cdots $};  \draw (-12.5, 10.6) node{$\vdots$};
  \draw (-14.5, 11.5) node{$\mQ_1$}; \draw (-13.5,11.5) node{$\cdots$};  \draw (-12.5,11.5) node{$\mQ_1$};

  \draw[dashed](-11,10.5)--(-7,10.5);

\draw (-6,9) grid (-3,12);
 \draw (-5.5,9.5) node{$\mQ_1$}; \draw (-4.5, 9.5) node{$\cdots$};  \draw (-3.5, 9.5) node {$\mQ_1$};
 \draw (-5.5, 10.6) node{$\vdots$};  \draw (-4.5, 10.5) node {$\cdots $};  \draw (-3.5, 10.6) node {$\vdots$};
  \draw (-5.5,11.5) node{$\mQ_1$}; \draw (-4.5,11.5) node{$\cdots$};  \draw (-3.5, 11.5) node {$\mQ_1$};

\draw (-3,9) grid (0,12);
 \draw (-2.5, 9.5) node {$\mQ_1$}; \draw (-1.5, 9.5) node {$\cdots$};  \draw (-0.5, 9.5) node{$\mQ_1$};
 \draw (-2.5, 10.6) node {$\vdots$};  \draw (-1.5, 10.5) node{$\cdots $};  \draw (-0.5, 10.6) node{$\vdots$};
  \draw (-2.5,11.5) node{$\mQ_1$}; \draw (-1.5, 11.5) node{$\cdots$};  \draw (-0.5,11.5) node{$\mQ_1$};

 \draw (0,9) grid (3,12);
 \draw (0.5, 9.5) node {$\mQ_1$}; \draw (1.5, 9.5) node{$\cdots$};  \draw (2.5, 9.5) node{$\mQ_1$};
 \draw (0.5, 10.6) node {$\vdots$};  \draw (1.5, 10.5) node{$\cdot $};  \draw (2.5, 10.6) node{$\vdots$};
  \draw (0.5,11.5) node {$\mQ_1$}; \draw (1.5, 11.5) node {$\cdots$};  \draw (2.5,11.5) node{$\mQ_1$};
  
  \draw (3,9) grid (6,12);
 \draw (3.5, 9.5) node{$\mQ_1$}; \draw (4.5, 9.5) node{$\cdots$};  \draw (5.5, 9.5) node{$\mQ_1$};
 \draw (3.5, 10.6) node{$\vdots$};  \draw (4.5, 10.5) node{$\cdots $};  \draw (5.5, 10.6) node{$\vdots$};
  \draw (3.5,11.5) node{$\mQ_1$}; \draw (4.5,11.5) node{$\cdots$};  \draw (5.5,11.5) node{$\mQ_1$};
  
   \draw (6,9) grid (9,12);
 \draw (6.5, 9.5) node{$\mQ_1$}; \draw (7.5, 9.5) node{$\cdots$};  \draw (8.5, 9.5) node{$\mQ_1$};
 \draw (6.5, 10.6) node{$\vdots$};  \draw (7.5, 10.5) node{$\cdots $};  \draw (8.5, 10.6) node{$\vdots$};
  \draw (6.5,11.5) node{$\mQ_1$}; \draw (7.5,11.5) node{$\cdots$};  \draw (8.5,11.5) node{$\mQ_1$};

          \draw[dashed](10, 10.5)--(14, 10.5);
  
  \draw (15,9) grid (18,12);
 \draw (15.5, 9.5) node{$\mQ_1$}; \draw (16.5, 9.5) node{$\cdots$};  \draw (17.5, 9.5) node{$\mQ_1$};
 \draw (15.5, 10.6) node{$\vdots$};  \draw (16.5, 10.5) node{$\cdots $};  \draw (17.5, 10.6) node{$\vdots$};
  \draw (15.5, 11.5) node{$\mQ_1$}; \draw (16.5,11.5) node{$\cdots$};  \draw (17.5,11.5) node{$\mQ_1$};

    \draw (18,9) grid (21,12);
 \draw (18.5, 9.5) node{$\mQ_1$}; \draw (19.5, 9.5) node{$\cdots$};  \draw (20.5, 9.5) node{$\mQ_1$};
 \draw (18.5, 10.6) node{$\vdots$};  \draw (19.5, 10.5) node{$\cdots $};  \draw (20.5, 10.6) node{$\vdots$};
  \draw (18.5,11.5) node{$\mQ_1$}; \draw (19.5,11.5) node{$\cdots$};  \draw (20.5,11.5) node{$\mQ_1$};


\draw[dashed](-15,-8) --(-15, -7); 
\draw[dashed](-12,-8) --(-12, -7);
\draw[dashed](-9,-8) --(-9, -7);
\draw[dashed](-6,-8) --(-6, -7);
\draw[dashed](-3,-8) --(-3, -7);
\draw[dashed](0,-8) --(0, -7);
\draw[dashed](3,-8) --(3, -7);
\draw[dashed](6,-8) --(6, -7);
 \draw[dashed](9,-8) --(9, -7);
 \draw[dashed](12,-8) --(12, -7);
 \draw[dashed](15,-8) --(15, -7);
 \draw[dashed](18,-8) --(18, -7);

  
 \end{tikzpicture}
}\\
\vspace{0.35cm}
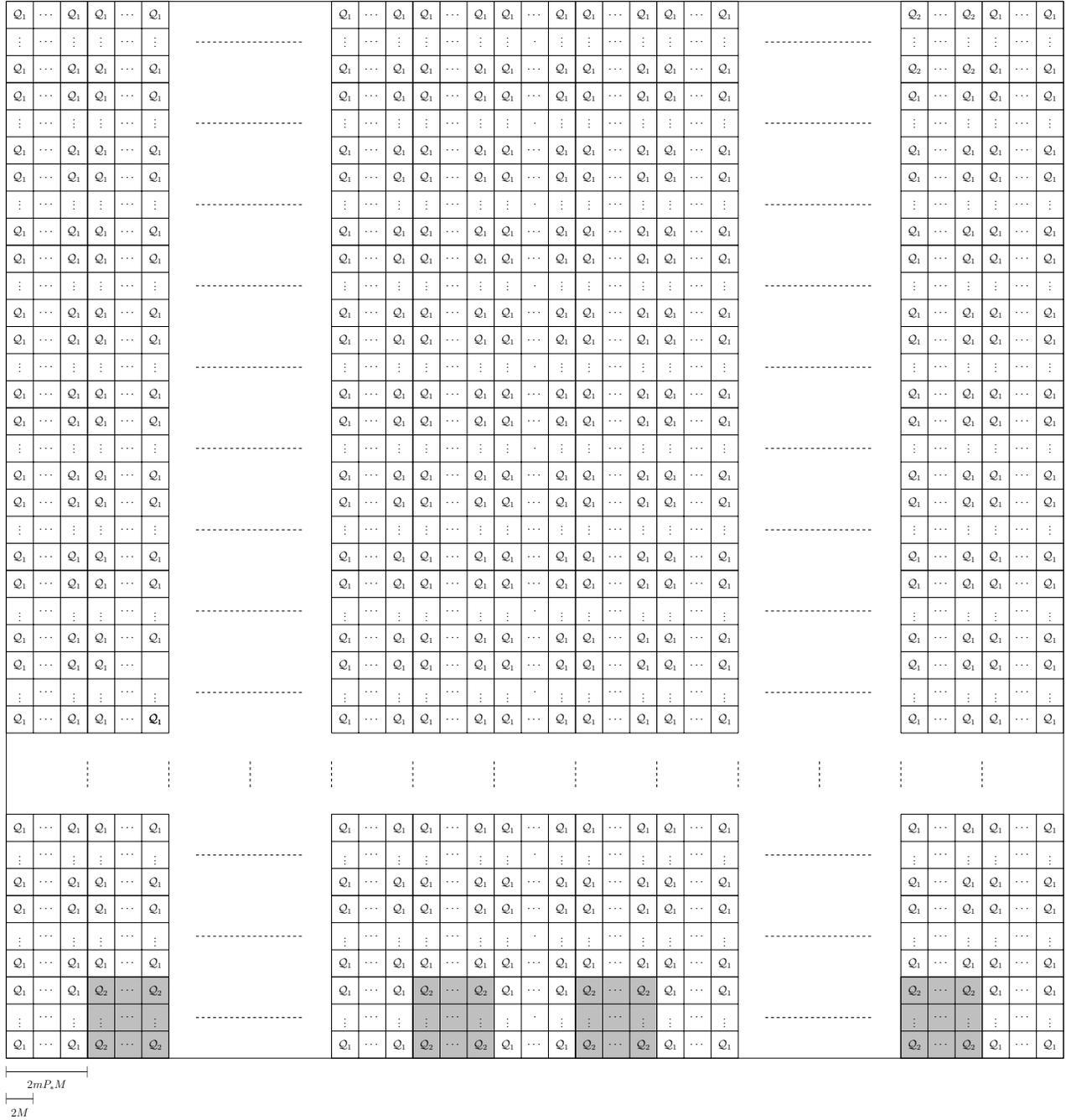
Figure 1: Building $\mathcal{Q}_1^1$ starting with $\mathcal{Q}_1^0 = \mathcal{Q}_1$ 
and $\mathcal{Q}_2^0 = \mathcal{Q}_2$.
\end{center}


\vspace{0.64cm}
\vspace{0.1cm}

It is now easy to construct a non-rectifiable, repetitive Delone set. Indeed, 
let $(L_n)$ be a sequence of numbers $\geq 1$ going to infinity. Start with 
the square patches $\mQ_{1,1}$ and $\mQ_{1,2}$ illustrated below:

\vspace{5mm}

\begin{center}
 \resizebox{9cm}{!}{%
\begin{tikzpicture}


\draw (0,4) node{$\bullet$};\draw (1,4) node{$\bullet$};\draw (2,4) node{$\bullet$};\draw (3,4) node{$\bullet$};\draw (4,4) node{$\bullet$};

\draw (0,3) node{$\bullet$};\draw (1,3) node{ };\draw (2,3) node{$\bullet$};\draw (3,3) node{ };\draw (4,3) node{$\bullet$};

\draw (0,2) node{$\bullet$};\draw (1,2) node{ };\draw (2,2) node{$\bullet$};\draw (3,2) node{ };\draw (4,2) node{$\bullet$};

\draw (0,1) node{$\bullet$};\draw (1,1) node{ };\draw (2,1) node{$\bullet$};\draw (3,1) node{ };\draw (4,1) node{$\bullet$};

\draw (0,0) node{$\bullet$};\draw (1,0) node{$\bullet$};\draw (2,0) node{$\bullet$};\draw (3,0) node{$\bullet$};\draw (4,0) node{$\bullet$};

\draw(2.1, -1) node{$\mQ_{1,1}$}; 


\draw (8,4) node{$\bullet$};\draw (9,4) node{$\bullet$};\draw (10,4) node{$\bullet$};\draw (11,4) node{$\bullet$};\draw (12,4) node{$\bullet$};

\draw (8,3) node{$\bullet$};\draw (9,3) node{$\bullet$ };\draw (10,3) node{$\bullet$};\draw (11,3) node{ $\bullet$ };\draw (12,3) node{$\bullet$};

\draw (8,2) node{$\bullet$};\draw (9,2) node{$\bullet$ };\draw (10,2) node{$\bullet$};\draw (11,2) node{$\bullet$ };\draw (12,2) node{$\bullet$};

\draw (8,1) node{$\bullet$};\draw (9,1) node{$\bullet$ };\draw (10,1) node{$\bullet$};\draw (11,1) node{$\bullet$ };\draw (12,1) node{$\bullet$};

\draw (8,0) node{$\bullet$};\draw (9,0) node{$\bullet$};\draw (10,0) node{$\bullet$};\draw (11,0) node{$\bullet$};\draw (12,0) node{$\bullet$};

\draw(10.1, -1) node{$\mQ_{1,2}$}; 
\end{tikzpicture}}
\end{center}
\vspace{-0.6cm}
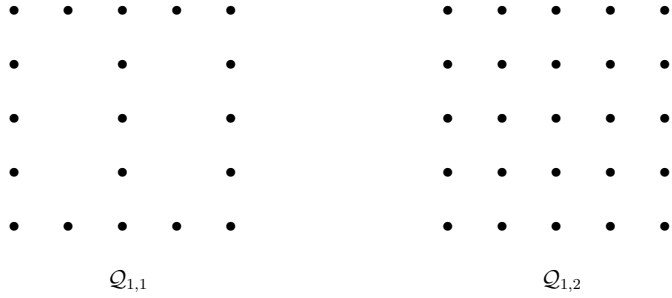
\begin{center} Figure 2: The initial patches $\mQ_{1,1}$ and $\mQ_{1,2}$. \end{center}


\vspace{0.4cm}

Next, proceed inductively: assuming we are given the patches $\mQ_{n,1} \!=:\! \mQ_1$ 
and $\mQ_{n,2} \!=:\! \mQ_2$, we let $\mQ_{n+1,1} \!:=\! \mQ_{1}^{new}$ and 
$\mQ_{n+1,2} \!:=\! \mQ_{2}^{new}$, where we have implemented the preceding 
procedure to obtain new patches for the constant $L_n$. This construction 
fits into that of Lemma \ref{lema-repetitivo}, except for that the patches 
$\mP_{n,1}, \mP_{n,2}$ that are involved do not correspond to $\mQ_{n,1}, \mQ_{n,2}$, 
respectively, but to the lower-left corners of these. (This is due to that the matchings 
above were made by identifying left to right sides, and lower to upper sides.) Hence, we have 
a repetitive Delone set $\mD$ containing copies of $\mQ_{n,1}$ and $\mQ_{n,2}$, for each 
$n \geq 1$. By Proposition \ref{no-bilipschitz}, $\mD$ cannot be $L_n$-bi-Lipschitz 
equivalent to $\mathbb{Z}^2$ for any $n \geq 1$. Since $L_n \to \infty$, the set 
$\mD$ is not bi-Lipschitz equivalent to $\mathbb{Z}^2$. 

\begin{rem} Clearly, the properties of being repetitive and non-rectifiable is not only valid 
for $\mD$ but also for all points in the closure of its orbit under the translation action. 
\end{rem}

We end this section with a brief discussion concerning the lack of linear repetitivity of our 
examples. Roughly, this amounts to saying that the ratio of the side-length of the new squares 
$\mathcal{Q}_1^{new},\mathcal{Q}_2^{new}$ compared to that of the original ones 
$\mathcal{Q}_1,\mathcal{Q}_2$ appearing along the construction tends to infinity at 
least along a subsequence. In our construction, this essentially comes from the condition 
$$N \geq N_* \geq \frac{10^{10} L^{10}}{(d-d')^4};$$
see Remark \ref{segunda-estimacion} and estimate (\ref{condition-kappa-3}).

Despite this, given an unbounded function $R' \! : (0,\infty) \to (0,\infty)$, we can artificially introduce 
steps in which the parameter $N$ does not satisfy (\ref{condition-kappa-3}) but just $N \!=\! 1$. Doing 
this infinitely many times, we obtain an infinite sequence of radii $r_k$ for which the repetitivity 
function $R$ satisfies $R(r_k) \leq R'(r_k)$. Notice that the resulting Delone set is still non-rectifiable, 
as an easy application of the triangular inequality shows that these steps do not obstruct the steps 
along which (\ref{condition-kappa-3}) is satisfied and that yield a contradiction to rectifiability. 
It is quite surprising that, actually, the choice $N=1$ allows requiring $R'$ just to be larger 
than some universal constant for infinitely many values, not necessarily being unbounded.

It is more interesting trying to obtain explicit estimates on the growth of the sequence $r_k$ provided 
$R'$ has some nice behaviour, for instance, if it grows faster than linearly. If we pay attention only 
to Remark \ref{segunda-estimacion}, then this requires $r_k$ to be of the order of a product 
$C^k L_k \cdots L_2 L_1$ for some universal constant $C > 1$ (where $L_n$ is the sequence of 
Lipschitz constants to be discarded so that $L_n \to \infty$) provided $R' (r_k)$ is larger than  $C L_k^{10} r_k$. 
Indeed, the value of the denominator $(d-d')$ can be bounded from below by a universal 
positive constant all along the construction. (Notice that (\ref{condition-kappa-2}) does not alter this issue.) 

Nevertheless, there is another condition, namely (\ref{condition-kappa-3}), which is more restrictive. Indeed, the corresponding 
expressions $(d_1' - d_1)$ and $(d_2 - d_2')$ that do appear in the denominators cannot be bounded from below by an 
universal constant. They can, however, be bounded from below by a sequence of positive numbers with finite sum smaller 
than 1, as for instance $1/cn^{1+\alpha}$ for an appropriate constant $c$. This allows controlling the value of the repetitivity 
function $R$ along a sequence of radii having the order of $c L_k \ldots L_1 (k!)^{1+\alpha}$, where $L_n \to \infty$.

In any case, we think that many steps of our construction can be improved. 
In this direction, it is very tempting thinking that, 
given a function $R'$ growing faster than linearly, a non-rectifiable Delone set 
exists so that $R (r)\leq R'(r)$ holds for all large-enough $r$. Besides, it is natural to think that linear repetitiveness 
is not the optimal condition to ensure rectifiability, and that some finite moment condition on the function $R$ 
still should imply this. We do not see, however, any potential application of these seemingly hard questions. 

\vspace{0.21cm}


\noindent{\bf IV. {\em Combining patches to get unique ergodicity.}} 
As we already mentioned, for a repetitive Delone set, unique ergodicity is equivalent to that 
all patches appearing in the tiling have a well-defined asymptotic density. This is closely related to 
\cite[Theorem 3.3]{solomyak}, but there is an anternate way to see this. Namely, since the Delone sets that we consider are subsets of $\ZZ^2$, we can use Wiener's unique ergodicity criterion for 
$\ZZ^d$-subshifts (see for example \cite{Lim12}). That is, the $\ZZ^2$-action on the closure of the orbit of $\mathcal{D}$ 
is uniquely ergodic if and only if for every $\mathcal{D}'$ in this orbit-closure and every patch $\mathcal{Q}$ of 
$\mathcal{D}'$, the limit  
$$
\lim_{n\to \infty}\frac{ \mbox{number of occurrences of } \mathcal{Q} \mbox{ in } \mathcal{D}' \big|_{[-n,n]^d} }{(2n+1)^d}.
$$
exists and is independent of $\mathcal{D}'$. Moreover, by the proof of \cite[Theorem 3.3]{solomyak}, this condition needs to 
be chequed only for a single $\mathcal{D}'$, say for $\mathcal{D}$. We claim that in the schema of Lemma \ref{lema-repetitivo}, 
this is the case whenever all asymptotic densities of occurrences of the patches $\mathcal{P}_{m,i}$ as blocks in 
$\mathcal{P}_{n,i}$, with $n \to \infty$, are equal to $1/2$. 

\vspace{0.1cm}

\begin{lem} {\em Assume that all asymptotic densities of occurrences of the patches $\mathcal{P}_{m,i}$ 
as blocks in $\mathcal{P}_{n,i}$, with $n \to \infty$, are equal to $1/2$. Then the limit}
\begin{equation}\label{a--probar}
\lim_{n\to \infty}\frac{ \mbox{number of occurrences of } \mathcal{Q} \mbox{ in } \mathcal{D} \big|_{[-n,n]^d} }{(2n+1)^d}
\end{equation}
{\em exists for every patch $\mathcal{Q}$ appearing in $\mathcal{D}$.}
\end{lem}

\noindent{\bf Proof.} First, an easy application of a Whitney like decomposition shows that the limit (\ref{a--probar}) 
exists if and only if the limit 
\begin{equation}\label{new-to-prove}
\lim_{n\to \infty}\frac{ \mbox{number of occurrences of } \mathcal{Q} \mbox{ in } \mathcal{P}_{n,j} }
{\big| \mathcal{P}_{n,j} \big|}
\end{equation}
exists and is independent of $j$. To show that the last condition holds, 
for each $m \geq 1$ and $j \in \{1,2\}$, denote 
$$d_{m,j} := \frac{ \mbox{number of occurrences of } \mathcal{Q} \mbox{ in } \mathcal{P}_{m,j} }
{\big| \mathcal{P}_{m,j} \big|}.$$
Besides, denote $d_{i,j}^{m \to n}$ the density in which the patch $\mathcal{P}_{m,i}$ appears as a block 
of $\mathcal{P}_{n,j}$. Let $\ell$ be the side length of $\mathcal{Q}$, and assume that $m$ is large enough 
so that $\ell$ is smaller than the side length of each $\mathcal{P}_{m,j}$. If we divide a given square 
$\mathcal{P}_{n,j}$ into the blocks $\mathcal{P}_{m,1}$ and $\mathcal{P}_{m,2}$, we have 
$$d_{m,1} d_{1,1}^{m \to n} + d_{m,2} d_{2,1}^{m \to n} 
\leq d_{n,1} 
\leq d_{m,1} d_{1,1}^{m \to n} + d_{m,2} d_{2,1}^{m \to n} + 
       \frac{2\ell}{\mbox{ side length of } \mathcal{P}_{m,j}}$$ 
and 
$$d_{m,1} d_{1,2}^{m \to n} + d_{m,2} d_{2,2}^{m \to n} 
\leq d_{n,2} 
\leq d_{m,1} d_{1,2}^{m \to n} + d_{m,2} d_{2,2}^{m \to n} + 
       \frac{2 \ell}{\mbox{ side length of } \mathcal{P}_{m,j}}.$$
Indeed, the left-side inequalities are obvious, while in the right-side expression, the extra term appears because of the 
possibility that a copy of $\mathcal{Q}_{m,j}$ overlaps with two different blocks (in either left-to-rigth or bottom-to-top direction). 

By hypothesis, for all fixed $m$ and each $i,j$ in $\{1,2\}$, the value of $d_{i,j}^{m\to n}$ converges to $1/2$ as $n \to \infty$. 
It thus follows from the inequalities above that given $\varepsilon > 0$, there exist integers $m$ and $N$ such that for all 
$n \geq N$, 
$$\frac{d_{m,1} + d_{m,2}}{2} - \varepsilon \leq d_{n,1} \leq \frac{d_{m,1} + d_{m,2}}{2} + \varepsilon,$$
$$\frac{d_{m,1} + d_{m,2}}{2} - \varepsilon \leq d_{n,2} \leq \frac{d_{m,1} + d_{m,2}}{2} + \varepsilon.$$
In particular, there exists a sequence of integers $n_k$ such that for each $k$ and all $n \geq n_{k+1}$, 
\begin{equation}\label{unila1}
\frac{d_{n_k,1} + d_{n_k,2}}{2} - \frac{1}{k} \leq d_{n,1} \leq \frac{d_{n_k,1} + d_{n_k,2}}{2} + \frac{1}{k},
\end{equation}
$$\frac{d_{n_k,1} + d_{n_k,2}}{2} - \frac{1}{k}  \leq d_{n,2} \leq \frac{d_{n_k,1} + d_{n_k,2}}{2} + \frac{1}{k}.$$
As a consequence, both $(d_{n,1})$ and $(d_{n,2})$ are Cauchy sequences, hence they converge to certain 
limits $d_1$ and $d_2$, respectively. Letting $k \to \infty$ in (\ref{unila1}) along $n = n_{k+1}$, we obtain 
$$\frac{d_1 + d_2}{2} \leq d_1 \leq \frac{d_1 + d_2}{2},$$
hence $d_1 = d_2$, as desired. 
$\hfill\square$

\vspace{0.275cm}

In order to guarantee the hypothesis of the preceding lemma and hence 
proving unique ergodicity of the translation action on the orbit-closure of $\mathcal{D}$, 
we will need to crucially modify the preceding construction. As above, we will only use two types 
of patches at each step, and we will start with the same (lower-left corners of the) patches illustrated
in Figure 2. Therefore, the density of points in the resulting Delone set will be equal to 
$$\frac{1}{2} \cdot \frac{16}{16} + \frac{1}{2} \cdot \frac{10}{16} = \frac{13}{16}.$$

We begin by introducing the {\em transition matrices}
$$\mathcal{A}^{n \to n+1} = \big( d_{i,j}^{n \to n+1} \big),$$
where, as before, $d_{i,j}^{n\to n+1}$ stands for the density in which the patch 
$\mP_{n,i}$ appears in $\mP_{n+1,j}$, with $i,j$ in $\{ 1,2 \}$. 
If we let 
$$\mathcal{A}^{m \to n} 
= 
\mathcal{A}^{m \to m+1} \mathcal{A}^{m+1 \to m+2} \cdots \mathcal{A}^{n-1 \to n}$$
and denote $d_{i,j}^{m\to n}$ the entries of $\mathcal{A}^{m \to n}$, then $d_{i,j}^{m\to n}$ 
represents, as before, the density in which the patch $\mP_{m,i}$ appears in $\mP_{n,j}$. 
In particular, if $d_i$ is the density of points in the starting patch $\mP_{1,i}$, where 
$i \in \{1,2\}$, then the density of points in $\mP_{n,i}$ equals 
$$d_{n,i} := d_0 \cdot d_{0,i}^{1\to n} + d_1 \cdot d_{1,i}^{1 \to n}.$$

To simplify, we will only work with transition matrices of the form
\begin{equation}\label{tipo-de-matriz}
\mathcal{A}^{n \to n+1} = \left(
\begin{array}
{cc}
1/2 +\delta_n & 1/2 -\delta_n \\
1/2 -\delta_n & 1/2 +\delta_n \\
\end{array}
\right). 
\end{equation}
To deal with these matrices, we will strongly use the identity
\begin{equation}\label{multiplicacion}
\left(
\begin{array}
{cc}
1/2 +\alpha & 1/2 -\alpha \\
1/2 -\alpha & 1/2 +\alpha \\
\end{array}
\right) 
\left(
\begin{array}
{cc}
1/2 +\beta & 1/2 -\beta \\
1/2 -\beta & 1/2 +\beta \\
\end{array}
\right) 
= 
\left(
\begin{array}
{cc}
1/2+ 2 \alpha \beta & 1/2 - 2 \alpha \beta \\
1/2 - 2 \alpha \beta & 1/2 + 2 \alpha \beta \\
\end{array}
\right).
\end{equation}
This shows, in particular, that for $\alpha,\beta$ between $0$ and $1/2$, 
the $\|\cdot\|_{\infty}$ distance between 
$$\left(
\begin{array}
{cc}
1/2 +\alpha & 1/2 -\alpha \\
1/2 -\alpha & 1/2 +\alpha \\
\end{array}
\right) 
\left(
\begin{array}
{cc}
1/2 +\beta & 1/2 -\beta \\
1/2 -\beta & 1/2 +\beta \\
\end{array}
\right)
\qquad{ \mbox{and} } \qquad 
\left(
\begin{array}
{cc}
1/2 & 1/2 \\
1/2 & 1/2 \\
\end{array}
\right)$$
is less than or equal to $2 \beta$ times 
the $\|\cdot\|_{\infty}$ distance between  
$$\left(
\begin{array}
{cc}
1/2 +\alpha & 1/2 -\alpha \\
1/2 -\alpha & 1/2 +\alpha \\
\end{array}
\right)
\qquad{ \mbox{ and } } \qquad 
\left(
\begin{array}
{cc}
1/2 & 1/2 \\
1/2 & 1/2 \\
\end{array}
\right).$$

As before, we start the construction with the (lower-left corners of the) patches illustrated 
in Figure 2. Next, we proceed by induction: assuming that we have constructed the patches 
$\mQ_{n,1} =: \mQ_1$ and $\mQ_{n,2} =: \mQ_2$, we let $\mQ_{n+1,1}' := \mQ_1^{new}$ 
and $\mQ_{n+1,2}' := \mQ_2^{new}$, where we have implemented the construction of new 
patches of the preceding paragraph for the constant $L_n := n$ and 
$$d_2' := d_{n,2} - \frac{d_{n,2} - d_{n,1}}{3}, \quad d_1' = d_{n,1} + \frac{d_{n,2} - d_{n,1}}{3}.$$ 
By construction, this procedure consists of a certain number $\ell \!=\! \ell_n$ of intermediate 
steps along which all transition matrices are of the form (\ref{tipo-de-matriz}). In particular, 
by the previous discussion, we did not lose any amount of closeness to the desired limit matrix   
$\left(
\begin{array}
{cc}
1/2 & 1/2 \\
1/2 & 1/2 \\
\end{array}
\right)$ 
along this construction. 

Next, to construct $\mQ_{n+1,1}$ and $\mQ_{n+1,2}$, we mix (and match) 
together $\mQ_{n+1,1}'$ and $\mQ_{n+1,2}'$ appropriately, as shown in Figure 3: 

\vspace{5mm}

\begin{center}
 \resizebox{11cm}{!}{%
\begin{tikzpicture}

\draw[step=1.5cm] (0,0) grid (4.5,4.5);
\draw (0.7,3.7) node{$\mQ_{n+1,2}'$};\draw (2.2, 3.7) node{$\mQ_{n+1,1}'$};\draw (3.7, 3.7) node{$\mQ_{n+1,2}'$}; 

\draw (0.7, 2.2) node{$\mQ_{n+1,1}'$};\draw (2.2,2.2) node{$\mQ_{n+1,1}'$};\draw (3.7,2.2) node{$\mQ_{n+1,1}'$}; 
  
\draw (0.7, 0.7) node{$\mQ_{n+1,2}'$};\draw (2.2,0.7) node{$\mQ_{n+1,1}'$};\draw (3.7,0.7) node{$\mQ_{n+1,2}'$}; 

\draw(2.3, -0.8) node{$\mQ_{n+1,1}$}; 


\draw [step=1.5cm] (7.5,0) grid (12,4.5);
\draw (8.2, 3.7) node{$\mQ_{n+1,1}'$};\draw (9.7,3.7) node{$\mQ_{n+1,2}'$};\draw (11.2, 3.7) node{$\mQ_{n+1,1}'$}; 

\draw (8.2,2.2) node{$\mQ_{n+1,2}'$};\draw (9.7,2.2) node{$\mQ_{n+1,2}'$};\draw (11.2,2.2) node{$\mQ_{n+1,2}'$}; 
  
\draw (8.2,0.7) node{$\mQ_{n+1,1}'$};\draw (9.7,0.7) node{$\mQ_{n+1,2}'$};\draw (11.2,0.7) node{$\mQ_{n+1,1}'$}; 
 
\draw(9.7, -0.8) node{$\mQ_{n+1,2}$}; 
\end{tikzpicture}}
\end{center}

\vspace{-0.7cm}

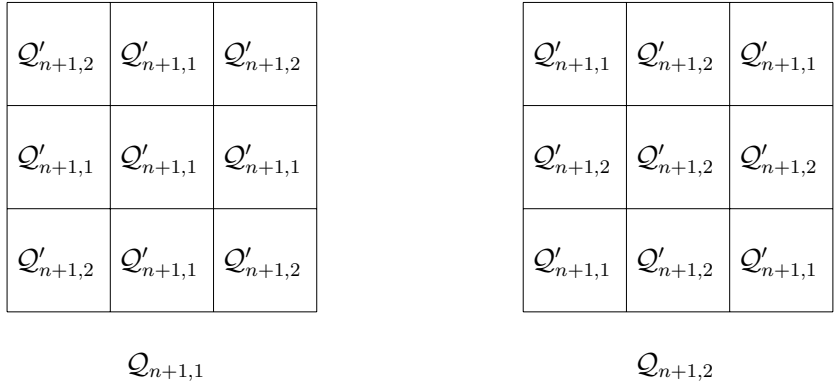
\begin{center} Figure 3: Building $\mQ_{n+1,1}$ and $\mQ_{n+1,2}$ 
starting with $\mQ_{n+1,1}'$ and $\mQ_{n+1,2}'$. \end{center}

\vspace{0.45cm}

Letting $\mP_{n+1,i}$ be the lower-left corner of $\mQ_{n+1,i}$, with $i \in \{1,2\}$, we 
have that the density of $\mP_{n+1,1}'$ inside $\mP_{n+1,1}$ (resp. $\mP_{n+1,2}$) equals 
$5/9 = 1/2 + 1/18$ (resp. $4/9 = 1/2 - 1/18$). Similarly, the density of $\mP_{n+1,2}'$ inside 
$\mP_{n+1,1}$ (resp. $\mP_{n+1,2}$) equals $4/9 = 1/2 - 1/18$ (resp. $5/9 = 1/2 + 1/18$). 

By the construction, the transition matrix from the patch $\mP_{n,i}$ (hence of any $\mP_{m,i}$, 
with $m \leq n$) to each $\mP_{n+1,j}$ is of the form (\ref{tipo-de-matriz}). In particular, we have 
$d_{n,2} > d_{n,1}$ for all $n$. Moreover, due to (\ref{multiplicacion}), the $\| \cdot \|_{\infty}$ 
distance between any transition matrix $\mathcal{M}^{m \to n+1}$ and 
$\left(
\begin{array}
{cc}
1/2 & 1/2 \\
1/2 & 1/2 \\
\end{array}
\right)$
is less than or equal to $\frac{1}{9}$ times the $\| \cdot \|_{\infty}$ 
distance between the transition matrix $\mathcal{M}^{m \to n}$ and 
$\left(
\begin{array}
{cc}
1/2 & 1/2 \\
1/2 & 1/2 \\
\end{array}
\right) \!.$
Letting $n$ go to infinity, this yields the desired convergence. 

\vspace{0.5cm}


\noindent{\bf V. {\em Prescribing the (shape of the) set of invariant probability measures.}} There 
are many ways to realize arbitrary Choquet simplices, one of which is given by the next lemma. 
For the statement, given positive integers $k, q$, we let $\triangle(k,q)$ be the convex hull 
of the set of vectors $\frac{e_1}{q}, \ldots, \frac{e_k}{q}$, where $\{ e_1,\ldots, e_k \}$ 
stands for the canonical orthonormal basis of $\mathbb{R}^k$.

\vspace{0.1cm}

\begin{lem}\label{Choquet}
Let $\mathcal{K}$ be a Choquet simplex,  let $(q_n)$ be an increasing sequence of positive 
integers such that each $q_n$ divides $q_{n+1}$, and let $(r_n)$ be a sequence of positive 
integers satisfying $r_n \sqrt{q_{n}}~<~\sqrt{q_{n+1}}$. Then there exists 
a sequence $(A_n)$ of $k_n \!\times k_{n+1}$ matrices with positive integer entries such that, 
passing to a subsequence of $(q_n)$ if necessary (as well as to the corresponding subsequence 
of $(r_n)$), we have:
\begin{itemize}
\item [{\bf (K1)}] $k_1 = \max\{3,d\}$ if $\mathcal{K}$ has dimension $d$, 
and $k_1 = 3$ if $\mathcal{K}$ has infinite dimension. 
\item[{\bf (K2)}] $k_{n}\geq 3$, for all $n$;
\item[{\bf (K3)}] $A_n (1, j) = 1$, for every $j \in [\![ 1, k_{n+1} ]\!]$; 
\item[{\bf (K4)}] $\sum_{i=1}^{k_n} A_n(i,j) = \frac{q_{n+1}}{q_{n}}$, for every $j \in [\![ 1,k_{n+1} ]\!]$;
\item[{\bf (K5)}] $\min \big\{ A_{n}(i,j) \!: 2 \leq i \leq k_n, 1 \leq j \leq k_{n+1} \big\} \geq k_{n+1}$;
\item[{\bf (K6)}] $\min \big\{ A_{n}(i,j) \!: 2 \leq i \leq k_n, 1 \leq j \leq k_{n+1} \big\} \geq r_n\sqrt{q_{n+1}}$; 
\item[{\bf (K7)}] $\mathcal{K}$ is affine homeomorphic to the inverse limit
$$
\lim_{\leftarrow n} \big( \triangle(k_n,q_{n}), A_n \big) := 
\Big\{ (u_n) \in \prod_{n\geq 1}\triangle(k_n, p_{n}) 
\!: A_n (u_{n+1}) = u_n, \mbox{ for all } n \Big\}.
$$
\end{itemize}
\end{lem}

\begin{proof} 
By \cite[Lemmas 9 and 13]{CP13}, there exists a sequence $(B_{\ell})$ of 
$k_{\ell} \times k_{\ell +1}$ matrices with positive integer entries such that 
$k_{\ell} \geq 2$ for all $\ell$ and verifying (K4), (K7) and 
$$k_{\ell +1} \leq \min \big\{ B_{\ell}(i,j) \!: 1 \leq i \leq k_\ell, 1 \leq j \leq k_{\ell+1} \big\}.$$
Next, notice that since all matrix entries are $\geq 1$, using (K3) we easily obtain by induction 
that for every $m>m'$, all $i \in [\![ 1, k_{m'} ]\!]$ and all $j \in [\![ 1, k_{m+1} ]\!]$,
$$
B_{m'} \cdots B_m ( i,j ) \geq \frac{q_{m+1}}{q_{m'+1}}.
$$ 
Let $\ell_1:=1$, and given $\ell_n$, define $\ell_{n+1}$ so that 
$q_{\ell_{n+1}}>  (1+q_{\ell_n+1})^2 \hspace{0.06cm} r_{\ell–n}^2$. Then, the 
matrices $\tilde{A}_n := B_{\ell_n}\cdots B_{\ell_{n+1}-1}$ satisfy (K4), (K7), and 
$$\min \big\{ \tilde{A}_{n}(i,j) \!: 1 \leq i \leq k_{\ell_n}, 1 \leq j \leq k_{\ell_{n+1}} \big\} 
\geq \max \{ k_{\ell_{n+1}},  r_{\ell_n} \sqrt{q_{\ell_{n+1}}} \}.$$
Finally, defining $A_n$ as the $(k_{\ell_n}+1)\times (k_{\ell_{n+1}}+1)$ matrix with columns
$$
\big( A_n(\cdot,1) \big) 
= \big( A_n(\cdot, 2) \big) 
=\left( \begin{array}{l}
                                                                                1\\
                         \tilde{A}_n(1,1)-1\\
\tilde{A}_n(2,1)\\
\vdots\\
\tilde{A}_n(k_{\ell_n},1)
                                                                                 \end{array}
\right), \qquad 
\big( A_n(\cdot,k+1) \big) 
= \left( \begin{array}{l}
1\\
\tilde{A}_n(1,k)-1\\
\tilde{A}_n(2,k)\\
\vdots\\
\tilde{A}_n(k_{\ell_n},k)\\                                            
\end{array}
    \right),
$$
where $2 \leq k \leq k_{\ell_n}$, we have that 
all properties (K2), (K3), (K4), (K5) and (K6) are satisfied 
with respect to the subsequences $(q_{\ell_n})$, $(r_{\ell_n})$. 
By \cite[Lemmas 1 and 2]{CP13}, property (K7) is also satisfied. 
Finally, property (K1) follows from \cite[Lemma 9]{CP13} and the 
proof of \cite[Lemma 13]{CP13} (this is independent on the choice of $(q_n)$).
\end{proof}

\vspace{0.2cm}

In all what follows, we will assume that $\mathcal{K}$ is not reduced to a singleton.  In other words, 
we will search for the construction of a non uniquely ergodic translation action over the orbit of a 
non-rectifiable Delone set, the uniquely ergodic case having been settled in the previous section. 

\begin{lem}\label{extremos} With the notation above, assume that 
$\mathcal{K}$ is not reduced to a singleton. Then there exist positive integers 
$m' \geq m$ and $i_0 \in [\![ 1,k_{m} ]\!]$ as well as real numbers $\bar{d}>\bar{d}'$ in $]0,1[$ such 
that for every $n \geq m'$, there exist  $j_{n+1},  j_{n+1}'$ in $[\![ 1, k_{n+1} ]\!]$ satisfying
$$
\frac{A_{m}\cdots A_n(i_0,j_{n+1})}{q_{n+1}} \geq \bar{d} 
\qquad \mbox{ and } \qquad  
\frac{A_{m}\cdots A_n(i_0,j_{n+1}')}{q_{n+1}} \leq \bar{d}'.
$$
\end{lem}

\begin{proof} Since $\mathcal{K}$ has at least two extreme points, there exist $(u_n), (v_n)$ in 
\hspace{0.02cm} $\lim_{\leftarrow n}(\triangle(k_n,q_{n}), A_n)$ such that for some positive integers 
$m$ and $i \in [\![ 1, k_{m} ]\!]$, the $i^{th}$-coordinates $u_{m,i}$ and $v_{m,i}$ of $u_{m}$ and 
$v_{m}$, respectively, are different. For each $n > m$,  we set
$$\alpha_n = \max\left \{ \left |\frac{A_{m}\cdots A_n (i,r)}{q_{n+1}} 
- \frac{A_{m}\cdots A_n(i,s)}{q_{n+1}}\right | \!\! : r,s \mbox{ in } [\![ 1, k_{n+1} ]\!] \right \}.$$
Suppose for a contradiction that there exists a subsequence $(\alpha_{n_{\ell}})$ converging to zero.  Then 
for every $j \in [\![ 1, k_{n_{\ell}+1} ]\!]$, there exists $\delta_{\ell,j} \in [-\alpha_{n_{\ell}}, \alpha_{n_{\ell}}]$ 
such that
$$
\frac{A_{m}\cdots A_{n_{\ell}}(i,j)}{q_{n_{\ell}+1}} 
=
\frac{A_{m}\cdots A_{n_{\ell}}(i,1)}{q_{n_{\ell}+1}}+\delta_{\ell,j}.
$$
Therefore,
\begin{eqnarray*}
u_{m,i} 
&=& \sum_{j=1}^{k_{n_{\ell}+1}}\frac{A_{m}\cdots A_{n_{\ell}}(i,j)}{q_{n_{\ell}+1}}q_{n_{\ell}+1}u_{n_{\ell},j}\\
&=&  \sum_{j=1}^{k_{n_{\ell}+1}} \left( \frac{A_{m}\cdots 
          A_{n_{\ell}}(i,1)}{q_{n_{\ell}+1}} +\delta_{\ell,j} \right) q_{n_{\ell}+1}u_{n_{\ell},j}\\
&=& \frac{A_{m}\cdots A_{n_{\ell}}(i,1)}{q_{n_{\ell}+1}} 
+  \sum_{j=1}^{k_{n_{\ell}+1}}\delta_{\ell,j}q_{n_{\ell}+1}u_{n_{\ell},j}
\end{eqnarray*}
and 
$$v_{m,i}=\frac{A_{m}\cdots A_{n_{\ell}}(i,1)}{q_{n_{\ell}+1}} 
+  \sum_{j=1}^{k_{n_{\ell}+1}}\delta_{\ell,j}q_{n_{\ell}+1} v_{n_{\ell},j}.$$
Thus we get
$$\big| u_{m,i} - v_{m,i} \big| 
\leq \sum_{j=1}^{k_{n_{\ell}+1}}|\delta_{\ell,j}|q_{n_{\ell}+1}(u_{n_{\ell},j}+v_{n_{\ell},j})
\leq 2 \alpha_{n_{\ell}},$$
which contradicts the fact that $u_{m,i} \neq v_{m,i}$.
\end{proof}

\vspace{0.1cm}

Let $\mathcal{K}$ be a Choquet simplex not reduced to a singleton, and let $(p_n)$ be a 
sequence of positive integers such that $p_1 = \max \{4,d\}$ for $\mathcal{K}$ 
$d$-dimensional, $p_1 = 4$ for $\mathcal{K}$ infinite-dimensional, and such 
that for every $n \!\geq\! 1$, one has $p_{n+1}=2(l_n+1)p_n$ for an integer $l_n\geq 1$.  
Let $(A_n)$ be a sequence of $k_n \times k_{n+1}$ matrices with positive integer entries 
verifying  the properties of Lemma \ref{Choquet} with respect to $q_n := p_n^2$. Let $m' \geq m$, 
$i_0 \in [\![ 1, k_{m} ]\!]$,  $d>d'$ in $]0,1[$ and $j_{n+1},  j_{n+1}'$ in $[\![ 1, k_{n+1} ]\!]$, 
be as in Lemma \ref{extremos}, where $n \geq m'$. Observe that we can (and we will) assume that 
$m=1$ and that both $j_{n+1}, j_{n+1}'$ are $\geq 2$ (the last assumption because the first two 
columns of each matrix $A_n$ are equal). Let $(r_n)$ be a sequence of positive integers such that 
$r_n p_n <p_{n+1}$, for all $n$.  

We set $F_1 := [\![0,p_1-1]\!]^2$, and for $n \geq 1$, we let 
$$
F_{n+1} := \bigcup_{v\in [\![-l_{n}-1,l_{n}]\!]^2}(F_{n}+p_{n}v).
$$
Next, we define the patch 
$$\mP_{1,i_0} := F_1\setminus \{(p_1-1, p_1-1)\}.$$ 
For $k\in [\![1, k_1]\!]\setminus \{i_0\}$, the patch $\mP_{1,k}$ is defined as (see Figure 4 below)
$$\mathcal{P}_{1,k} := \big\{ (i,j) \in F_1 \!:  i \mbox{ is even} \big\} 
\cup \big\{ (i,j) \in F_1 \!: j=0 \big\} \cup \big\{ (1,k) \big\}.$$  

\vspace{0.1mm}

\begin{center}
 \resizebox{8.5cm}{!}{%
\begin{tikzpicture}

\draw (0,5) node{$\bullet$};\draw (1,5) node{$\bullet$ };\draw (2,5) node{$\bullet$};\draw (3,5) node{$\bullet$ };\draw (4,5) node{$\bullet$};\draw (5,5) node{ };

\draw (0,4) node{$\bullet$};\draw (1,4) node{$\bullet$};\draw (2,4) node{$\bullet$};\draw (3,4) node{$\bullet$};\draw (4,4) node{$\bullet$};\draw (5,4) node{$\bullet$};

\draw (0,3) node{$\bullet$};\draw (1,3) node{$\bullet$ };\draw (2,3) node{$\bullet$};\draw (3,3) node{$\bullet$ };\draw (4,3) node{$\bullet$};\draw (5,3) node{$\bullet$};

\draw (0,2) node{$\bullet$};\draw (1,2) node{ $\bullet$};\draw (2,2) node{$\bullet$};\draw (3,2) node{$\bullet$ };\draw (4,2) node{$\bullet$};\draw (5,2) node{$\bullet$};

\draw (0,1) node{$\bullet$};\draw (1,1) node{ $\bullet$};\draw (2,1) node{$\bullet$};\draw (3,1) node{$\bullet$ };\draw (4,1) node{$\bullet$};\draw (5,1) node{$\bullet$};

\draw (0,0) node{$\bullet$};\draw (1,0) node{$\bullet$};\draw (2,0) node{$\bullet$};\draw (3,0) node{$\bullet$};\draw (4,0) node{$\bullet$};\draw (5,0) node{$\bullet$};

\draw(2.1, -1) node{$\mP_{1,i_0}$}; 

\draw (8,5) node{$\bullet$};\draw (9,5) node{ };\draw (10,5) node{$\bullet$};\draw (11,5) node{ };\draw (12,5) node{$\bullet$}; \draw (13,5) node{ }; 

\draw (8,4) node{$\bullet$};\draw (9,4) node{ };\draw (10,4) node{$\bullet$};\draw (11,4) node{};\draw (12,4) node{$\bullet$};\draw (13,4) node{}; 

\draw (8,3) node{$\bullet$};\draw (9,3) node{$\bullet$ };\draw (10,3) node{$\bullet$};\draw (11,3) node{  };\draw (12,3) node{$\bullet$};\draw (13,3) node{ }; 

\draw (8,2) node{$\bullet$};\draw (9,2) node{  };\draw (10,2) node{$\bullet$};\draw (11,2) node{ };\draw (12,2) node{$\bullet$};\draw (13,2) node{}; 

\draw (8,1) node{$\bullet$};\draw (9,1) node{  };\draw (10,1) node{$\bullet$};\draw (11,1) node{ };\draw (12,1) node{$\bullet$};\draw (13,1) node{ }; 

\draw (8,0) node{$\bullet$};\draw (9,0) node{$\bullet$};\draw (10,0) node{$\bullet$};\draw (11,0) node{$\bullet$};\draw (12,0) node{$\bullet$};\draw (13,0) node{$\bullet$}; 

\draw(10.1, -1) node{$\mP_{1,k}$ for $k=3 \neq i_0$}; 
\end{tikzpicture}}
\end{center}
\vspace{-0.6cm}
\begin{center} 
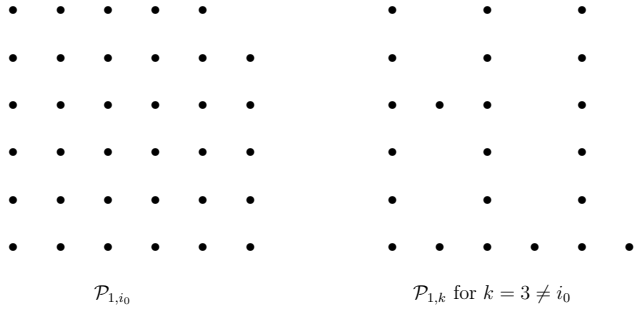
Figure 4: The patches $\mP_{1,k}$ for $p_1=6$, $k_1\geq 3$ and $i_0\neq 3$.
\end{center}

\vspace{0.3cm}

We next proceed to define patches $\mathcal{P}_{2,1},\ldots, \mathcal{P}_{2,k_2}$ in 
$\{ 0,1 \}^{F_2}$ satisfying:
\begin{itemize}

\item $\mP_{2,j} \cap \big( F_{1}+(l_1 p_1,l_1 p_1) \big)=\mP_{1,1}$, for each $j \in [\![ 1,k_2 ]\!]$ 
(that is, the upper-right corner of each $\mP_{2,j}$ is a copy of $\mP_{1,1}$);

\item  $\mathcal{P}_{2,j}\cap (F_{1}+vp_1)$ belongs to $\{\mathcal{P}_{1,1},\ldots, \mathcal{P}_{1,k_{1}}\}$, 
for every $j \!\in\! [\![ 1, k_{2} ]\!]$ and all $v \!\in\! [\![ -l_{1}-1,l_{1} ]\!]^2$;

\item  For all $i \in [\![ 1, k_{1} ]\!]$ and all $j \in [\![ 1, k_{2} ]\!]$, the number of vectors 
$v\in  [\![ -l_{1}-1,l_{1} ]\!]^2$ such that $\mathcal{P}_{2,j}\cap (F_{1}+vp_{1}) = \mP_{1,i}$ 
equals $A_{1}(i,j)$.
\end{itemize}
In order to check that it is possible to obtain $k_2$ different patches satisfying these three properties, just 
observe that the number of different ways to define a single patch $\mP_{2,k}$ satisfying all of them equals
$$
\frac{\big( \sum_{i=2}^{k_1}A_1(i,k) \big)!}{A_1(2,k)!\cdots A_1(k_1,k)!}
\geq \min\{A_1(i,j) \!: 2\leq i\leq k_1, 1\leq j\leq k_2\} \geq k_2.$$ 

Now, suppose that for $n\geq 2$, we have defined a collection $\mathcal{P}_{n,1},\ldots, \mathcal{P}_{n,k_n}$ 
of different patches in $\{0,1\}^{F_n}$. We will next proceed to define $k_{n+1}$ different patches  
$\mP_{n+1,1},\ldots, \mP_{n+1,k_{n+1}}$ in $\{0,1\}^{F_{n+1}}$ such that for all 
$k \in [\![ 1,k_{n+1} ]\!]$, the following properties are satisfied (see Figure 5):
\begin{itemize}
\item[{\bf (P1)}] $\mP_{n+1,k}\cap \big( F_n+(l_n p_n, l_n p_n ) \big) = \mP_{n,1}$;
\item[{\bf (P2)}]  For all $s\in [\![-l_{n}-1,l_{n}]\!]$ and $r\in [\![-l_{n}-1,-l_{n}+r_n-2]\!]$ , it holds
$$
\mP_{n+1,k}\cap \big( F_n+(sp_n, rp_n) \big) = \left\{ \begin{array}{lll}
                                                     \mP_{n,j_n} & \mbox{ if } &  
\Big[ \frac{s p_{n+1}}{r_n} \Big] \mbox{ is even,}\\
                                   \mP_{n,j_n'} & \mbox{ if } & 
\Big[ \frac{s p_{n+1}}{r_n} \Big] \mbox{ is odd;}
\end{array}\right.
$$
\vspace{-0.43cm}
\item [{\bf (P3)}] $\mathcal{P}_{n+1,k}\cap (F_{n}+vp_{n})$ belongs to 
$\{\mathcal{P}_{n,1},\ldots, \mathcal{P}_{n,k_{n}}\}$, for every $v\in [\![-l_{n}-1,l_{n}]\!]^2$;
\item[{\bf (P4)}]  The number of $v\in  [\![-l_{n}-1,l_{n}]\!]^2$ such that 
$\mathcal{P}_{n+1,k}\cap (F_{n}+vp_{n})=\mP_{n,i}$ equals $A_{n}(i,k)$.
\end{itemize}

\vspace{0.1mm}

\begin{center}
 \resizebox{8.5cm}{!}{%
\begin{tikzpicture}
\draw[style=dashed]  (0,0) rectangle (18,18);
\draw[fill=gray!50, dashed] (0,0) rectangle (3,3); 
 \draw[style=dashed]  (0,0) grid (3,3);
 \draw (0.5, 0.5) node{$\mP_{n,j_n}$};   \draw (1.5, 0.5) node{$\cdots$};  \draw (2.5, 0.5) node{$\mP_{n,j_n}$};
  \draw (0.5, 1.5) node{$\vdots$};            \draw (1.5, 1.5) node{$\cdots$};    \draw (2.5, 1.5) node{$\vdots$};
    \draw (0.5, 2.5) node{$\mP_{n,j_n}$};    \draw (1.5, 2.5) node{$\cdots$};     \draw (2.5, 2.5) node{$\mP_{n,j_n}$};
 
 \draw[style=dashed] (3,0) grid (6,3);
 \draw (3.5, 0.5) node{$\mP_{n,j_n'}$};   \draw (4.5, 0.5) node{$\cdots$};  \draw (5.5, 0.5) node{$\mP_{n,j_n'}$};
  \draw (3.5, 1.5) node{$\vdots$};            \draw (4.5, 1.5) node{$\cdots$};    \draw (5.5, 1.5) node{$\vdots$};
    \draw (3.5, 2.5) node{$\mP_{n,j_n'}$};    \draw (4.5, 2.5) node{$\cdots$};     \draw (5.5, 2.5) node{$\mP_{n,j_n'}$};
 
\draw[fill=gray!50, dashed] (12,0) rectangle (15,3); 
\draw[style=dashed]  (12,0) grid (15,3);
\draw (12.5, 0.5) node{$\mP_{n,j_n}$};   \draw (13.5, 0.5) node{$\cdots$};  \draw (14.5, 0.5) node{$\mP_{n,j_n}$};
  \draw (12.5, 1.5) node{$\vdots$};            \draw (13.5, 1.5) node{$\cdots$};    \draw (14.5, 1.5) node{$\vdots$};
    \draw (12.5, 2.5) node{$\mP_{n,j_n}$};    \draw (13.5, 2.5) node{$\cdots$};     \draw (14.5, 2.5) node{$\mP_{n,j_n}$};

\draw[style= dashed]  (15,0) grid (18,3);
 \draw (15.5, 0.5) node{$\mP_{n,j_n'}$};   \draw (16.5, 0.5) node{$\cdots$};  \draw (17.5, 0.5) node{$\mP_{n,j_n'}$};
  \draw (15.5, 1.5) node{$\vdots$};            \draw (16.5, 1.5) node{$\cdots$};    \draw (17.5, 1.5) node{$\vdots$};
    \draw (15.5, 2.5) node{$\mP_{n,j_n'}$};    \draw (16.5, 2.5) node{$\cdots$};     \draw (17.5, 2.5) node{$\mP_{n,j_n'}$};

\draw[fill=gray!30, dashed]  (17,17) rectangle (18,18);
\draw (17.5, 17.5) node{$\mP_{n,1}$};

\draw[style= dashed] (6.5,0.5) -- (11.5,0.5);
\draw[style= dashed] (6.5,1.5) -- (11.5,1.5);
\draw[style= dashed] (6.5,2.5) -- (11.5,2.5);
 \end{tikzpicture}}
\end{center}
\vspace{-0.6cm}
\begin{center} 
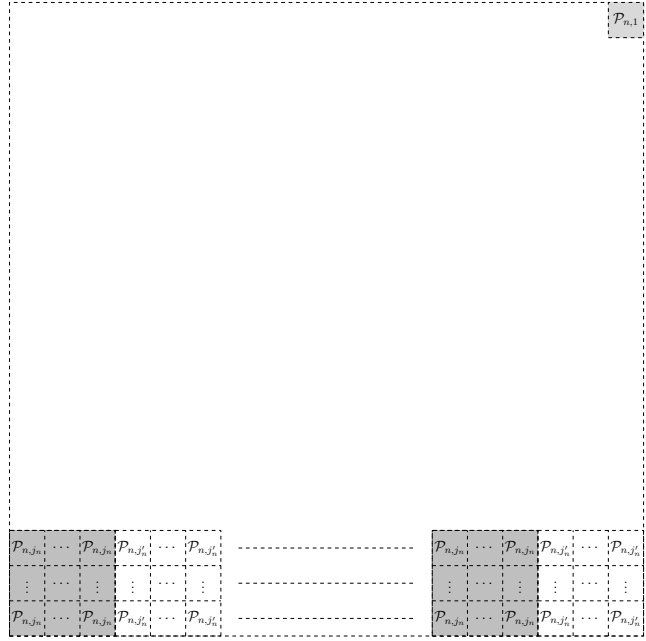
Figure 5: Building the patches $\mP_{n+1,k}$: the white part must be filled according to the rules (P3) 
and (P4), and the dashed lines stand for that we do not overlap patches as in the previous sections.
\end{center}

Notice that (P1) and (P2) completely determine how to fill 
\hspace{0.02cm} $\frac{p_{n+1}}{p_n}r_n \!+\! 1$ \hspace{0.02cm} translated copies 
of $F_n$. We thus need to fill, in different ways, the remaining (free) \hspace{0.02cm}
$\frac{p_{n+1}^2}{p_n^2}-\frac{p_{n+1}}{p_n}r_n-1$ \hspace{0.02cm} translated 
copies of $F_n$ in a way that (P4) is satisfied. To do this, notice that if 
$p_{n+1}$ is sufficiently large,  namely
\begin{equation}\label{ojo}
p_{n+1}>\frac{(k_n-1)p_n^2}{k_n-2}\left(\frac{r_n}{p_n}+1 \right),
\end{equation}
then
$$
(k_n-2)\frac{p_{n+1}^2}{p_n^2}  >   (k_n-1)\frac{p_{n+1}r_n}{p_n}+(k_n-1)p_{n+1},
$$
which implies
\begin{eqnarray*}
(k_n-1)\left( \frac{p_{n+1}^2}{p_n^2}-  \frac{p_{n+1}r_n}{p_n}- 1\right) 
\!\!& > &\!\!
(k_n-1)\left( \frac{p_{n+1}^2}{p_n^2}-  \frac{p_{n+1}r_n}{p_n}- p_{n+1}\right) 
>  \frac{p_{n+1}^2}{p_n^2}
>  \sum_{i=2}^{k_n}A_n(i,j)\\
&\geq& \!\! (k_n-1)\min\{A_n(i,j): 1\leq i\leq k_n, 1\leq j\leq k_{n+1}\}.
\end{eqnarray*}
Using this and (K5), we obtain 
$$
\frac{p_{n+1}^2}{p_n^2}-\frac{p_{n+1}r_n}{p_n}-1 >  k_{n+1}.
$$
Next, we notice that among the free translated copies of $F_n$, the number of those   
that have to be filled by copies of $\mP_{n,j_n}$ (resp. $\mP_{n,j_n'}$) equals 
$$A_{n}(j_n,k)-\frac{r_np_{n+1}}{2p_n}\geq r_np_{n+1}-r_n\frac{p_{n+1}}{2p_n}>0 \quad 
\left( \mbox{resp. } A_{n}(j_n',k)-\frac{r_np_{n+1}}{2p_n}\geq r_np_{n+1}-r_n\frac{p_{n+1}}{2p_n}>0 \right).$$
This easily allows producing patches $\mP_{n+1,1},\ldots, \mP_{n+1,k_{n+1}}$ that do satisfy 
(P4) and differentiate one from each other in the places where we put some patches 
$\mP_{n,j_n}, \mP_{n,j_n'}$ in a fixed family of $k_n$ free translated copies of $F_n$.
 
\vspace{0.2cm}
 
Having defined all patches $\mathcal{P}_{i,j}$, let us now consider the family of sets
$$X_n := \big\{ D\subseteq \mathbb{Z}^2 \!: D \cap (F_n+v) \in \{\mP_{n,1},\ldots, \mP_{n,k_n}\}, 
\mbox{ for every } v\in p_n\ZZ^2 \big\}.$$
It is clear that $(X_n)$ is a nested sequence of nonempty compact sets, hence their intersection is nonempty. 
Moreover, every element in this intersection is a Delone set that satisfies the $2\mathbb{Z}$-property. Fix 
such a set $\mathcal{D}$, and let $X$ be the closure of its orbit with respect to the translation action of 
$\mathbb{Z}^2$ (equivalently, of $\mathbb{R}^2$). For $n\geq 1$ and $k \in [\![ 1, k_n ]\!]$, we set
$$
C_{n,k} := \{D \in X \!: D \cap F_n=\mP_{n,k}\}.
$$
By the construction, 
$$U_n := \{C_{n,k}+v: 1\leq k\leq k_n, v\in F_n\}$$ 
is a clopen covering of $X$.  We claim that it is actually a partition of $X$. 
To show this, let us first consider the case of $U_1$. For all $D\!\in\! X_1$ and all 
$v \!\in\! p_1\ZZ^2$, the intersection $D \cap (F_1+v)$ belongs to $\{\mP_{1,1},\ldots, \mP_{1,k_1}\}$. 
If two atoms of $U_1$, say $C_{1,k} + v$ and $C_{1,k'}+v'$, have nonempty intersection, then 
letting $u := v - v'$, we have that $C_{1,k} + u$ intersects $C_{1,k'}$. Then, by looking at all 
possible intersections and having in mind the geometry of the patches $\mP_{1,k}$, one easily 
convinces that $u$ must belong to $p_1 \mathbb{Z}^2$. Since both $v$ and $v'$ lie in $F_1$, 
this implies that $u = 0$, hence $v = v'$, and finally $k = k'$. The proof for $(U_n)$ 
works by induction. Assuming that $U_{n-1}$ is a partition, a similar argument applies 
taking into account that the unique position in which $\mP_{n-1,1}$ appears in each 
patch $\mP_{n,k} $ is the upper-right corner. 

Next, let $\mu$ be an invariant probability measure for the translation action of 
$\mathbb{Z}^2$ on $X$.  We claim that the vectors of the $\mu$-measures, 
namely 
$$\mu_n := \big( \mu(C_{n,1}),\ldots, \mu(C_{n,k_n}) \big),$$ 
satisfy $\mu_n^T = A_n (\mu_{n+1}^T)$, for every $n\geq 1$. Indeed, we have
 \begin{eqnarray*}
 \mu(C_{n,i}) 
& = & \mu \Big( \bigcup_{k=1}^{k_{n+1}} \big\{ 
C_{n+1,k}+v \!: v\in F_{n+1}, C_{n+1,k}+v\subseteq C_{n,i} \big\} \Big)\\
       &=& \sum_{k=1}^{k_{n+1}} \big| \{v\in F_{n+1} \!:  
               C_{n+1,k}+v\subseteq C_{n,i}\} \big| \cdot \mu(C_{n+1,k})\\
       &=&\sum_{k=1}^{k_{n+1}} A_n(i,k) \cdot \mu(C_{n+1,k}),
 \end{eqnarray*}
which shows our claim.

We can thus consider the sequence $(\mu_n)$ as a point in the inverse limit 
$\lim_{\leftarrow n}(\triangle(k_n, p_n^2), A_n)$. Notice that the function 
$\mu \mapsto (\mu_n)$ from the set of invariant probability measures 
into the space $\lim_{\leftarrow n}(\triangle(k_n, p_n^2), A_n)$ is affine. 
We claim that it is a bijection. Indeed, on the one hand, given $(u_n)$ in 
$\lim_{\leftarrow n}(\triangle(k_n, p_n^2), A_n)$, we may produce a probability 
measure $\mu$ on $X$ by letting $\mu(C_{n,k}+v) = u_n(k)$, for every $k \in [\![ 1, k_n ]\!]$ 
and all $v\in F_n$. It is the not hard to check that $\mu$ is invariant under the translation action 
(see \cite[Lemma 5]{CP06}), thus showing the surjectivity of the map. 
On the other hand, to check that it is injective, consider the set
$$
X^* := 
\bigcup_{w\in \mathbb{Z}^2} 
\bigcap_{n\geq 1} 
\bigcup_{k=1}^{k_n}\bigcup_{v \in F_n \setminus F_n- w} 
\big( C_{n,k}+v \big). 
$$
This set contains all points of $X$ (if any) that are not separated by the partitions $(U_n)$. 
Indeed, if $D,D'$ are two such points, then for each $n \geq 1$ they belong to the same 
atom $C_{n,i_n} + v_n$ in $U_n$. If $D,D'$ are different, then there is $w \in \mathbb{Z}^2$ 
contained only in one of them.  Thus, $D + w$ and $D + w'$ differ at the origin, and therefore  
$C_{n,i_n} + v_n + w$ cannot be an atom of $U_n$. This implies that $v_n + w \notin F_n$, 
that is $v_n \in F_n \setminus F_n - w$, which shows our claim. 

Using the fact that $(F_n)$ is a F\o lner sequence, one can easily check that $\mu(X^*)=0$ 
for every invariant probability measure $\mu$. Indeed, for all $n \geq 1$ and all fixed 
$w \in \mathbb{Z}^2$, 
\begin{eqnarray*}
\mu \left( \bigcup_{k=1}^{k_n}\bigcup_{v \in F_n \setminus F_n- w} \big( C_{n,k}+v \big) \right)
&=& 
\sum_{k=1}^{k_n} \big| F_n \setminus F_n - w \big| \cdot \mu (C_{n,k}) \\
&=& 
\big| F_n \setminus F_n - w \big| \cdot \sum_{k=1}^{k_n} \mu (C_{n,k})\\
&=& \frac{\big| F_n \setminus F_n - w \big|}{|F_n|} 
    \hspace{0.1cm}  \underset{n\to \infty}{\longrightarrow}\hspace{0.1cm} 0,
\end{eqnarray*}
where the last equality follows from that $U_n$ is a partition of $X$. Thus, any given clopen 
set $C$ can be written as the union $C_1 \cup C_2$, where $C_1$ is a (countable) union 
of atoms of $(U_n)$ and $C_2$ is a subset of $X^*$. This shows that any probability 
measure $\mu$ on $X$ that is invariant under the translation action of $\mathbb{Z}^2$ 
is completely determined by the sequence $(\mu_n)$, thus showing the desired injectivity. 

\vspace{0.25cm}

We can now finish our construction. To do this, we consider the sequence $(p_n)$ defined by 
$p_1 := \max\{4,d\}$ in case $\mathcal{K}$ is $d$-dimensional , $p_1 := 4$ in case 
$\mathcal{K}$ is infinite-dimensional, and $p_{n+1} := 2 n! (p_n)^2$, for all 
$n \geq 1$. (This definition ensures property (\ref{ojo}).) Then we let $r_n := n!$, and 
we realize $\mathcal{K}$ as an inverse limit $\lim_{\leftarrow n} \big( \triangle(k_n,q_{n}), A_n \big)$, 
where $q_n := p_n^2$. Next, we perform the preceding construction for this realization. We thus obtain a 
Delone set $\mathcal{D}$ satisfying the $2\mathbb{Z}$-property and such that the set of invariant 
probability measures for the $\mathbb{Z}^2$-action on the closure of its orbit is affine isomorphic 
to $\mathcal{K}$. It remains showing that $\mathcal{D}$ is non-rectifiable. To do this, we will need 
the next 

\hspace{0.01cm}

\begin{lem}\label{densidades}  There exist $d>d'$ in $]0,1[$ such that for every $n>m'$,  
$$|\mP_{n,j_n}| \geq p_n^2 d>  p_n^2 d' \geq |\mP_{n,j'_n}|.$$
\end{lem}

\begin{proof} First notice that $|\mP_{1,i_0}|=p_1^2-1$ 
and that for every $k \in [\![1, k_1]\!]\setminus\{i_0\}$, 
$$|\mP_{1,k}|=\frac{p_1^2}{2}+\frac{p_1}{2}+1.$$
Thus for every $n\geq 1$ and $k\in [\![1, k_n]\!]$, we have 
\begin{eqnarray*}
|\mP_{n,k}| 
& = & 
A_1\cdots A_{n-1}(i_0,k)\left( p_1^2-1\right) 
+\left(\frac{p_n^2}{p_1^2}-A_1\cdots A_{n-1}(i_0,k)\right)\left( \frac{p_1^2}{2}+\frac{p_1}{2}+1\right)\\
&=& 
A_1\cdots A_{n-1}(i_0,k)\left( \frac{p_1^2}{2}-\frac{p_1}{2}-2 \right)+\frac{p_n^2}{p_1^2}\left( \frac{p_1^2}{2}+\frac{p_1}{2}+1\right)
\end{eqnarray*}
By Lemma \ref{extremos}, for every $n>m'$,
\begin{eqnarray*}
|\mP_{n,j_n}| 
&\geq& 
\bar{d} p_n^2\left( \frac{p_1^2}{2}-\frac{p_1}{2}-2 \right)+\frac{p_n^2}{p_1^2}\left( \frac{p_1^2}{2}+\frac{p_1}{2}+1\right)\\
& >& 
\bar{d}'p_n^2\left( \frac{p_1^2}{2}-\frac{p_1}{2}-2 \right)+\frac{p_n^2}{p_1^2}\left( \frac{p_1^2}{2}+\frac{p_1}{2}+1\right)\\
&\geq& 
A_1\cdots A_{n-1}(i_0,j_n')\left( \frac{p_1^2}{2}-\frac{p_1}{2}-2 \right)+\frac{p_n^2}{p_1^2}\left( \frac{p_1^2}{2}+\frac{p_1}{2}+1\right)\\
&=& 
|\mP_{n,j_n'}|
\end{eqnarray*}
Thus, letting 
$$
d:= 
\bar{d} \left( \frac{p_1^2}{2}-\frac{p_1}{2}-2 \right)+\frac{1}{p_1^2}\left( \frac{p_1^2}{2}+\frac{p_1}{2}+1\right)
\quad \mbox{and} \quad
d':=
\bar{d}'\left( \frac{p_1^2}{2}-\frac{p_1}{2}-2 \right)+\frac{1}{p_1^2}\left( \frac{p_1^2}{2}+\frac{p_1}{2}+1\right)\!,
$$
we get the desired property.
\end{proof}

\vspace{0.1cm}

To conclude, we write $p_{n+1} = 2 p_n (n! p_n)$ and we refer to Proposition 
\ref{expandiendo} identifying $n! p_n$ with $M$ (which is a multiple of $P_* p_n$ for 
any prescribed $P_*$ provided $n$ is large enough) and $p_n$ with $N$. Then, an application 
of Proposition \ref{expandiendo} along the lines of the proof of Lemma~\ref{no-bilipschitz} 
allows showing that $\mD$ is not $L$-bi-Lipschitz equivalent to $\mathbb{Z}^2$ for any 
prescribed $L$, hence non rectifiable. 


\vspace{0.35cm}

\noindent{\bf Acknowledgments.} 
We would like to thank D. Coronel for his many hints and comments, 
and the anonymous referee for her/his useful remarks and corrections. 
Both authors where partially funded by the Anillo Research Project 1103 DySyRF. 
The first-named author was also funded by the Fondecyt Research Project 1140213. 
The second-named author acknowledges the CNRS (UMR 8628, Univ. d'Orsay) 
as well as the ERC starting grant 257110 ``RaWG''  for the support during 
the final stage of this work. He would also like to thank T.~Dymarz, 
A.~Erschler, P.~Py and R.~Tessera for their interest and useful discussions. 


\begin{footnotesize}


\vspace{0.3cm}

\noindent Mar\'{\i}a Isabel Cortez (maria.cortez@usach.cl)

\vspace{0.1cm}

\noindent Andr\'es Navas (andres.navas@usach.cl)

\vspace{0.2cm}

\noindent Dep. de Matem\'aticas, Fac. de Ciencia, Univ. de Santiago\\
\noindent Alameda 3363, Estaci\'on Central, Santiago, Chile\\

\end{footnotesize}


\begin{thebibliography}{Dillo 83}

\bibitem{alestalo} {\sc Alestalo, P.; Trotsenko, D. A.; V\"ais\"al\"a, J.} 
Linear Bilipschitz Extension Property. {\em Sibirsk. Mat. Zh.} {\bf 44} (1993), 
no. {\bf 6}, 1226-1238. Translation into English in {\em Siberian Mathematical 
Journal} {\bf 44} (1993), no. {\bf 6}, 959-968.

\bibitem{ACG} {\sc Aliste-Prieto, J.; Coronel, D.; Gambaudo, J.-M.} 
Linearly repetitive Delone sets are rectifiable. {\em Ann. Inst. H. Poincar\'e Anal. 
Non Lin\'eaire} {\bf 30} (2013), no. {\bf 2}, 275-290. 

\bibitem{BK1} {\sc Burago, D.; Kleiner, B.} 
Separated nets in Euclidean space and Jacobians of bi-Lipschitz maps. 
{\em Geom. Funct. Anal.} {\bf 8} (1998), no. {\bf 2}, 273-282. 

\bibitem{BK2} {\sc Burago, D.; Kleiner, B.}
Rectifying separated nets. {\em Geom. Funct. Anal.} {\bf 12} (2002), no. {\bf 1}, 80-92. 

\bibitem{CP06} {\sc  Cortez, M.I.; Petite, S.} G-odometers and their almost 1-1 extensions.  
{\em J. London Math. Soc.} {\bf 78} (2008), 1-20. 

\bibitem{CP13} {\sc  Cortez, M.I.; Petite, S.} Invariant measures and orbit equivalence for 
generalized Toeplitz subshifts. {\em Groups, Geometry, and Dynamics} {\bf 8} (2014), 1007-1045.

\bibitem{garber} {\sc Garber, A. I.} On equivalence classes of separated nets. 
{\em Modeling and Analysis of Information Systems} {\bf 16} (2009), no. {\bf 2}, 109-118.

\bibitem{gromov} {\sc Gromov, M.} Asymptotic invariants of infinite groups. In 
{\em Geometric group theory, Vol. 2} (Sussex, 1991), 1-295, London Math. 
Soc. Lecture Note Ser. {\bf 182}, Cambridge Univ. Press, Cambridge (1993). 

\bibitem{H} {\sc Haynes, A.; Kelly, M.; Weiss, B.} Equivalence relations on separated nets 
arising from linear toral flows. {\em Proc. London Math. Soc.} {\bf 109} (2014), 1203-1228.

 \bibitem{Lim12} {\sc Lima, Y.} $\Bbb Z^d$-actions with prescribed topological and ergodic 
properties. {\em Ergodic Theory Dynam. Systems.}  {\bf 32} (2012), no. {\bf 1}, 191-209.

\bibitem{magazinov} {\sc Magazinov, A. N.} The family of bi-Lipschitz classes of Delone 
sets in Euclidean space has the cardinality of the continuum. {\em Proc. of the Steklov. 
Inst. of Math.} {\bf 275} (2011), 87-98.

\bibitem{mcmullen} {\sc McMullen, C. T.} 
Lipschitz maps and nets in Euclidean space. 
{\em Geom. Funct. Anal.} {\bf 8} (1998), no. {\bf 2}, 304-314. 

\bibitem{S} {\sc Shechtman, D.; Blech, I.; Gratias, D.; Cahn, J. W.} Metallic phase with long range 
orientational order and no translational symmetry. {\em Phys. Review Letters} {\bf 53} 
(1984), no {\bf 20}, 1951-1954.

\bibitem{solomon} {\sc Solomon, Y.} Substitution tilings and separated nets with similarities 
to the integer lattice. {\em Israel J. of Math.} {\bf 181} (2011), 445-460.

\bibitem{solomyak} {\sc Solomyak, B.} 
Dynamics of self-similar tilings. 
{\em Ergodic Theory Dynam. Systems} {\bf 17} (1997), no. {\bf 3}, 695-738. 

\end{thebibliography}
\end{document}